\documentclass{amsart}
\usepackage{graphicx}
\usepackage{amssymb,amscd,amsthm,amsxtra}
\usepackage{latexsym}
\usepackage{epsfig}
\usepackage{mathtools}
\usepackage{esint}
\usepackage{color}

\usepackage{cfr-lm}
\usepackage[osf]{libertine}
\usepackage{hyperref}
\hypersetup{colorlinks,breaklinks,
            linkcolor=[rgb]{0,0,0},
            citecolor=[rgb]{0,0,0},
            urlcolor=[rgb]{0,0,0}}

\newtheorem{thm}{Theorem}[section]
\newtheorem{cor}[thm]{Corollary}
\newtheorem{lem}[thm]{Lemma}
\newtheorem{prop}[thm]{Proposition}
\theoremstyle{definition}
\newtheorem{defn}[thm]{Definition}
\theoremstyle{remark}
\newtheorem{rem}[thm]{Remark}
\numberwithin{equation}{section}
\newcommand{\be}{\begin{equation}}
\newcommand{\ee}{\end{equation}}

\newcommand{\abs}[1]{\left\vert#1\right\vert}

\newcommand{\R}{\mathbb R}
\newcommand{\N}{\mathbb N}

\newcommand{\eps}{\varepsilon}
\newcommand{\loc}{{{\tiny{\mbox{loc}}}}}

\newcommand{\p}{\partial}
\newcommand{\di}{\displaystyle}

\newcommand{\comment}[1]{}

\begin{document}

\title[A vectorial problem with thin free boundary]{A vectorial problem with thin free boundary}
\author{D. De Silva}
\address{Daniela De Silva  \newline \indent Department of Mathematics  \newline \indent  Barnard College, Columbia University,  \newline \indent  New York, NY 10027}
\email{\tt  desilva@math.columbia.edu}
\author[G. Tortone]{G. Tortone}\thanks{}
\address{Giorgio Tortone \newline \indent Dipartimento di Matematica
	\newline\indent
	Alma Mater Studiorum Universit\`a di Bologna
	\newline\indent
	 Piazza di Porta San Donato 5, 40126 Bologna}
\email{giorgio.tortone@unibo.it}

\thanks{Work partially supported by the ERC Advanced Grant 2013 n.\ 339958 Complex Patterns for Strongly Interacting Dynamical Systems - COMPAT, held by Susanna Terracini.}

\begin{abstract}

We consider the vectorial analogue of the thin free boundary problem introduced in \cite{CRS} as a realization of a nonlocal version of the classical Bernoulli problem. We study optimal regularity, nondegeneracy, and density properties of local minimizers. Via a blow-up analysis based on a Weiss type monotonicity formula, we show that the free boundary is the union of a ``regular" and a ``singular" part. Finally we use a viscosity approach to prove $C^{1,\alpha}$ regularity of the regular part of the free boundary.
\end{abstract}
\maketitle

\section{Introduction}

In this paper we are interested in the vectorial version of the minimization problem introduced in \cite{CRS}. Precisely,
given a regular open set $\Omega\subset \R^{n+1}$, we consider the vectorial free boundary problem
\be\label{min}
\text{min}\left\{\int_{\Omega}{\abs{\nabla G}^2\mathrm{d}X} + \mathcal{L}_n(\Omega \cap \{\abs{G}>0\}\cap \R^n)\colon G \in H^1(\Omega,\R^m),\, G=\Phi \text{ on }\partial \Omega \right\},
\ee
with boundary data $\Phi=(\varphi^1,\ldots,\varphi^m)$. Since the free boundary lies on the lower dimensional subspace $\{x_{n+1}=0\}$, such a problem is usually called a \emph{thin} free boundary problem. With a slight abuse of notation, whenever it does not create confusion, we will denote with $G$ both the local minimizer in $\R^{n+1}$ and its trace on $\R^n\times \{0\}$. Since we are developing a local analysis, it is not restrictive to assume $\Omega = B_1\subset \R^{n+1}$. Hence, given a ball $B\subset B_1\subset \R^{n+1}$ and the following class of admissible competitors (say $\Phi \in H^1(\Omega)$),
$$
\mathcal{K}(B)=\{G \in H^{1}(B,\R^m),\, G=\Phi \text{ on }\partial B\},
$$
we say that $G$ is a \emph{local minimizer} of
\be\label{min.functional}
\mathcal{J}(G,B_1) = \int_{B_1}{\abs{\nabla G}^2\mathrm{d}X} + \mathcal{L}_n(B_1 \cap \{|G|>0\}\cap \R^n ),
\ee
in $B_1$ if it minimizes $\mathcal J(G,B)$ in the class of competitors $\mathcal{K}(B)$, for every ball $B\subset B_1$. Similarly $G$ is a \emph{global minimizer} in $\R^{n+1}$ if $G$ is a local minimizer on every ball $B\subset \R^{n+1}$.\\ The scalar one-phase case $G=g \geq 0$, was first investigated by Caffarelli, Roquejoffre and Sire \cite{CRS} in relation with the theory of semi-permeable membranes and models where turbulence
or long-range interactions occur, for example in flame propagation and also in the propagation of surfaces of discontinuities (see \cite{CRS, Duvaut1976} and references therein). Moreover, the authors considered this free boundary problem as the local realization of the classical one-phase problem (also called the Bernoulli type problem) for the
fractional Laplacian.\\
They proved general properties (optimal
regularity, non-degeneracy and classification of global solutions) for local minimizers, corresponding to those proved by Alt and Caffarelli in their pioneering paper \cite{AC} for the standard one-phase Bernoulli  problem (see \cite{DFS,Vel19} for a comprehensive survey of the results in this setting).\\ A major step toward understanding the regularity theory for thin free boundaries was then obtained in \cite{DR}, where the first author and Roquejoffre introduced the notion of viscosity solution to the Euler-Lagrange equation associated to the minimization problem:
\begin{equation}\label{FBintro}\begin{cases}
\Delta g = 0, \quad \textrm{in $B_1^+(g):= B_1 \setminus \{(x,0) : g(x,0)=0\} ,$}\\
\frac{\p g}{\p t^{1/2}}= 1, \quad \textrm{on $F(g):= B_1 \cap \p_{\R^n}\{(x,0) : g(x,0)>0\}$},
\end{cases}\end{equation}
where
$$
\frac{\p }{\p t^{1/2}}g(x_0) = \lim_{(t,z)\to (0,0)} \frac{g(x_0+t\nu(x_0),x_{n+1})}{t^{1/2}}, \quad x_0 \in F(g),
$$
with $\nu(x_0)$ the unit normal to the free boundary $F(g)$ at $x_0$ pointing toward $\{g>0\}$.

In \cite{DR} it was proved that in any dimension if the free boundary $F(g)$ is sufficiently ``flat" then it is $C^{1,\alpha}$. Afterwards, in a series of paper \cite{DS1,DS2,DS3} the first author and Savin improved the known results by answering the question of whether Lipschitz free boundaries are smooth. By combining variational and nonvariational techniques, they also showed that local minimizers have smooth free boundary except possibly for a small singular set of Hausdorff dimension $n -3$.\\Recently in \cite{EKPSS} the authors used the Rectifiable-Reifenberg and quantitative stratification framework of Naber-Valtorta to prove Hausdorff measure and structure results for the singular set. We remark that in \cite{AP} the authors removed the sign assumption by considering a two-phase problem with thin free boundary, in the same spirit of the classical work of Alt-Caffarelli-Friedman \cite{ACF}. While in \cite{ACF} it is proved the existence of a common free boundary between the two phases, in \cite{AP} the authors showed that positive and negative phases are always separated by nontrivial dead core.

On the other hand, in \cite{CSY} the authors initiated the study of free boundary problems where several flows are involved, and interact whenever there is a phase transition. Indeed, this problem describes stationary thermal insulation, allowing a prescribed heat loss from the insulating layer.  Similarly, this set-up arose in population dynamics where several species coexist, and overflow the patches. While in the case of competitive systems it is well known that competition gives rise to the so-called junction points, under cooperation the solutions tend to congregate and to show a smooth free boundary.

In  \cite{CSY,MTV1}, the authors considered the classical Bernoulli problem,
\be\label{minloc}
\text{min}\left\{\int_{\Omega}{\abs{\nabla G}^2\mathrm{d}x} + \mathcal{L}_{n+1}(\Omega\cap \{|G|>0\})\colon G \in H^1(\Omega,\R^m),\, G=\Phi \text{ on }\partial \Omega \right\},
\ee
and applied a reduction method to reduce the problem to its scalar counterpart by assuming nonnegativity of the components. This assumption was successively removed in \cite{MTV2}. As expected, in this case the structure of the singular set changes and the set of branching points $\mathrm{Sing}_2(F(G))$ arises, as natural in two-phase problems (see also the recent work \cite{Guido}). \\
 Recently, in \cite{DT} we developed a vectorial viscosity approach to understand the regularity of the free boundary in \eqref{minloc}, which takes advantage of the fact that the norm $\abs{
G}$ is a viscosity subsolution to the scalar one-phase problem.

 We also remark that as pointed out in \cite{MTV1,KL1,KL2}, problem \eqref{minloc} is related to a class of shape optimization problems involving the eigenvalues of the Dirichlet Laplacian.\\

\subsection{Main results and organization of the paper.} In the first part of this paper, we consider the classical  questions of optimal regularity, non-degeneracy, and density estimates for local minimizers to \eqref{min.functional}. Then, we derive a Weiss-type monotonicity formula, which allows us to use a blow-up analysis and characterize global blow-up limits, which in turn leads to the definition of the regular and singular part of the free boundary. In the second part of the paper, we prove that local minimizers are viscosity solutions (see Section \ref{visco} for the precise definition) of the vectorial thin one-phase problem ($A_0>0$ a precise dimensional constant):
\be \begin{cases} \label{VOP}
\Delta G =0 & \text{in $B_1^+(|G|):= B_1 \cap \{(x,0) :\abs{G(x,0)}>0\}$;}\\
 \frac{\partial}{\partial t^{1/2}} |G|=A_0 & \text{on $F(G):=B_1 \cap \p_{\R^n}\{(x,0) \colon \abs{G(x,0)}>0\}.$}
\end{cases}\ee
Thus, the analysis of the regular part of the free boundary can be carried out with the viscosity methods developed in \cite{DR, DT}.
Combining the two parts we obtain the following main theorem.
\begin{thm}\label{mmm}
The problem \eqref{min} admits a solution $G \in H^1(B_1;\R^m)$. Moreover, any solution is locally $C^{0,1/2}$-H\"{o}lder continuous in $B_1$ and the set $\{|G|>0\}\cap \{x_{n+1}=0\}$ has a locally finite perimeter in $B_1\cap \{x_{n+1}=0\}$. More precisely, the free boundary $F(G)$ is a disjoint union of a regular part $\mathrm{Reg}(F(G))$ and one-phase singular set $\mathrm{Sing}(F(G))$:
\begin{enumerate}
  \item[1.] $\mathrm{Reg}(F(G))$ is an open subset of $F(G)$ and is locally the graph of a $C^{1,\alpha}$ function.
\item[2.] $\mathrm{Sing}(F(G))$ consists only of points in which the Lebesgue density of $\{|G|>0\}\cap \{x_{n+1}=0\}$ is strictly between $1/2$ and $1$. Moreover, there is $n^* \geq 3$ such that:
    \begin{itemize}
      \item if $n < n^*$, then $\mathrm{Sing}(F(G))$ is empty;
\item if $n = n^*$, then $\mathrm{Sing}(F(G))$ contains at most a finite number of isolated points;
\item if $n > n^*$, then the $(n - n^*)$-dimensional Hausdorff measure of $\mathrm{Sing}(F(G))$ is locally finite in $B_1\cap \{x_{n+1}=0\}$.
    \end{itemize}
\end{enumerate}
\end{thm}

 As already remarked, in the local analogue \cite{MTV2} the authors proved the existence of a closed set of locally finite $(n-1)$-Hausdorff measure of branching point $\mathrm{Sing}_2(F(G))$. This set
 consists only of points in which the Lebesgue density of
the positivity set $\{|G|>0\}$ is $1$ and the blow-up limits are linear functions: this blow-up analysis implies that cusps pointing inwards, might appear. While in the local case this feature is natural, in the thin case the picture changes.
Indeed, as pointed out earlier in the thin two-phase problem the positive and negative phases are always separated thus the problem reduces locally back to a one-phase problem.

The main theorem of the second part is independent of the minimization problem and it reads as follows.
\begin{thm}\label{main.visc}
Let $G$ be a viscosity solution to \eqref{VOP} in $B_1$. There exists a universal constant $\bar\eps>0$ such that if $G$ is $\bar \eps$-flat in the $(f,\nu)$-directions in $B_1$, i.e. for some unit directions $f\in \R^m, \nu \in \R^n$
$$
|G(X) - U(\langle x,\nu\rangle,x_{n+1}) f| \leq  \bar\eps \quad \text{in $B_1$,}
$$
and
$$
 |G|(x,0) \equiv 0 \quad \text{in $B_1 \cap \{\langle x, \nu \rangle  < - \bar\eps\}$},
 $$
 then $F(G)\in C^{1,\alpha}$ in $B_{1/2}$.
\end{thm}

In the viscosity setting, differently from the local case treated in \cite{DT}, the reduction from the vector valued  problem to the scalar one, is almost straightforward and hence most technical results leading to the proof of Theorem \ref{main.visc} follow from slight modifications of those in \cite{DR} (see Subsection 6.1).

\begin{rem}
As in the scalar case, in light of the extension facts related to the half-laplacian our theory applies, among others, to vectorial Bernoulli type problem involving non local energies, like, for instance the solutions to the following problem
(when $s = 1/2$):
$$
\text{min}\left\{\sum_{i=1}^m\int_{\R^{2n}}{\frac{\abs{g^i(x)-g^i(y)}^2}{\abs{x-y}^{n+2s}}\mathrm{d}x
\mathrm{d}y} + \mathcal{L}_n(\mathcal{B}_1 \cap \{\abs{G}>0\})\colon
\begin{aligned}
&g^i \equiv \varphi^i \text{on $\R^n\setminus \mathcal{B}_1$}\\ &\text{for $i=1,\dots,m$}
\end{aligned}\right\},
$$
where $\mathcal{B}_1=B_1\cap \{x_{n+1}=0\}, \Phi = (\varphi^1,\dots,\varphi^m) \in H^{1/2}(\R^n; \R^m)$ is the boundary data in $\R^n\setminus \mathcal{B}_1$. In this regard,
our results extend the theory of vectorial free boundary problem to the fractional case.\\
In the same spirit of \cite{MTV1,KL1,KL2}, since our methodologies are quite robust, we believe that our results can be extended to the case of almost minimizer of \eqref{min} and suitably adapted to the nonlocal shape optimization problem
$$
\mbox{min}\left\{\sum_{i=1}^m \lambda^s_i(A)\colon \text{$A\subset \Omega$ $s$-quasi open }, |A|\leq c\right\},\quad\text{with }\lambda_i^s(A)=\min_{u \perp E_{i-1}} \frac{[u]^2_{H^s(\R^N)}}{\|u\|^2_{L^2(\R^N)}},
$$
where $\Omega\subset \R^n$ open and bounded, $c<|\Omega|$ and $E_{i}\subset H^s_0(A)$ is the space spanned by the first $i$ eigenfunctions (see \cite{BRS} for more details in this direction). \\
We refer to \cite{DSalmost} for the theory of almost minimizers to the scalar thin-one phase problem.
\end{rem}

The paper is organized as follows:
in Section \ref{local} we study the local behavior of minimizers near the free boundary by answering the classical  questions of optimal regularity, non-degeneracy and density estimates for local minimizers. Then, in Section \ref{weiss.sect} we derive a Weiss-type formula which will allows in Section \ref{blow} to characterize global blow-up limits. The blow-up analysis of Section \ref{blow} will lead to the definition of the regular $\mathrm{Reg}(F(G))$ and singular part $\mathrm{Sing}(F(G))$ of the free-boundary. Finally, in the remaining part of the paper we will use a viscosity approach to obtain $C^{1,\alpha}$ smoothness of the regular part $\mathrm{Reg}(F(G))$. First, in Section \ref{visco} we introduce a vector valued analogue of the notion of viscosity solution of \cite{DR} and we prove that local minimizers are viscosity solution. In Section \ref{harnack.sect} we develop a vectorial Harnack inequality, which will be the basic tool for our analysis of the regular part of the free boundary. Finally, in Section \ref{final} we prove the improvement of flatness result, from which the $C^{1,\alpha}$ regularity of a flat free boundary follows by standard arguments.

\subsection{Notation.} From now on, we denote by $\{e_i\}_{i=1,\ldots,n}$ and $\{f^i\}_{i=1,\ldots,m}$ canonical basis in $\R^n$ and $\R^m$ respectively. Unit directions in $\R^n$ and $\R^m$ will be typically denoted by $e$ and $f$. The Euclidean norm in either space is denoted by $|\cdot|$ while the dot product is denoted by $\langle \cdot, \cdot \rangle$.\\
A point $X \in \R^{n+1}$ will be denoted by $X=(x,x_{n+1})\in \R^n\times \R$ and we will use the notation $x=(x',x_n)$. Moreover, a ball in $\R^{n+1}$ with radius $r>0$ centered at $X$ is denoted by $B_r(X)$ and for simplicity $B_r=B_r(0)$. Also, we use $\mathcal{B}_r = B_r\cap \{x_{n+1}=0\}$ to denote the $n$-dimensional ball in $\R^n\times \{x_{n+1}=0\}$.\\ We will often consider the following sets: let $g$ be a continuous non-negative function in $B_r$, then
\begin{align*}
B^+_r(g) &:= B_r \setminus \{(x,0)\colon g(x,0)=0\} \subset \R^{n+1}\\
\mathcal{B}^+_r(g) &:= B^+_r(g) \cap \{x_{n+1}=0\}\subset \R^{n}.
\end{align*} By abuse of notation, if $G=(g^1,\ldots,g^m)$ is a vector valued continuous function, we use $B_r^+(G), \mathcal B_r^+(G)$ in place of $B^+_r(|G|), \mathcal B^+_r(|G|)$ respectively.
Also, we will denote with $P$ and $L$ respectively the half-hyperplane $P:= \{X \in \R^{n+1}\colon x_n\leq 0, x_{n+1} =0 \}$ and $L:= \{X \in \R^{n+1}\colon x_n= 0, x_{n+1} =0 \}$.

In regard to the problem \eqref{FBintro}, we remark that if  $F(g)$ is $C^2$ then any function $g$ which is harmonic in $B^+_1(g)$ has an asymptotic expansion at a point $ x_0\in F(g),$
$$g(x,s) = \alpha(x_0) U((x-x_0) \cdot \nu(x_0), s)  + o(|x-x_0|^{1/2}+s^{1/2}).$$
Here $U(t,s)$ is the real part of $\sqrt{z}$ which
in the polar coordinates $$t= r\cos\theta, \quad  s=r\sin\theta, \quad r\geq 0, \quad -\pi \leq  \theta \leq \pi,$$ is given by
\begin{equation}\label{Unew}U(t,s) = r^{1/2}\cos \frac \theta 2. \end{equation}
Then, the free boundary condition in \eqref{FBintro} requires that $\alpha \equiv 1$ on $F(g).$

The function $U$ plays a fundamental role in our analysis.\\

Throughout the paper we will often used the invariance of a local minimizer of \eqref{min} with respect to translations and dilations. More precisely, fixed $X_0 \in F(G)$ we define as the blow-up sequence of $G$ centered at $X_0$ the family
\be\label{blow.up}
G_{X_0,r}(X)=\frac{1}{r^{1/2}}G(X_0+rX) 
\ee
with $r>0$. Indeed, for every $R>0$ we get
\be\label{rescale.J}
\mathcal{J}(G,B_R(X_0)) = r^n \mathcal{J}(G_{X_0,r},B_{R/r})).
\ee
Consequently, $G$ is a local minimizer of $\mathcal{J}(\cdot,B_R(X_0))$ if and only if $G_{X_0,r}$ is a local minimizer of $\mathcal{J}(\cdot, B_{R/r})$.
As in \cite{DR} it is not restrictive to reduce the analysis to the case of minimizers that are symmetric with respect to the $x_{n+1}$-variable. Indeed, if $G_e =(g^1_e,\dots,g^m_e)$ is the even part of $G$ with respect to $\{x_{n+1}=0\}$,
$$
g^i_e(x,x_{n+1})=\frac{g^i(x,x_{n+1})+g^i(x,-x_{n+1})}{2},\quad \text{for $i=1,\dots,m$}
$$
we get
$$
\mathcal{J}(G,B_1)=\mathcal{J}(G_e,B)+\int_{B_1}\abs{\nabla G_o}^2\mathrm{d}X,
$$
with $G_o=G-G_e$ the odd part of $G$ with respect to $\{x_{n+1}=0\}$. By the minimality of $G$, for every $V \in H^1(B;\R^m), V=G_e+G_o$ on $\partial B$ we have
$$
\mathcal{J}(G_e,B) \leq \mathcal{J}(V,B) -\int_{B_1}\abs{\nabla G_o}^2\mathrm{d}X \leq \mathcal{J}(V-G_o,B),
$$
which proves our claim once we noticed that $V-G_o=G_e$ on $\partial B$. Thus, throughout the paper this will be tacitly assumed.

\section{Local behavior of solutions}\label{local}
Since the existence of an optimal vector for problem \eqref{min} is nowadays standard (see \cite{CRS}[Proposition 3.2]), we start by focusing on the properties of local minimizers and obtain in this section, optimal regularity, non-degeneracy, and density estimates. We use the notation from Subsection 1.1.

\subsection{Optimal Regularity}
Our proof follows the lines of the scalar case in \cite{CRS} and it is based on an harmonic replacement of each components of the minimizing vector $G=(g^1,\ldots,g^m)$.

First, we make the following basic observation.

\begin{lem}\label{lem1}If $G$ is a local minimizer in $B_1$, then $g^i$ harmonic in $B_1 \setminus \{x_{n+1}=0\}.$\end{lem}
\begin{proof} Denote by $B^+_1:= B_1 \cap \{x_{n+1}>0\}$, and similarly $B_1^-:=B_1 \cap \{x_{n+1}<0\}$.
Let $\varphi$ be in  $C_0^\infty(B_1^+)$ and call $G^\pm_i:= (g^1,\ldots, g^i \pm \eps \varphi, \ldots, g^m)$, for any $i=1,\ldots,m$.  Thus, since $G \equiv G^\pm_i$ on $\{x_{n+1}=0\}$, the minimality of $G$,  $$\mathcal{J}(G, B_1) \leq \mathcal{J}(G^\pm_i, B_1)$$ implies $$\int_{B_1} |\nabla g^i|^2 \mathrm{d}X \leq \int_{B_1} |\nabla(g^i\pm\eps \varphi)|^2 dX$$ and hence $$\int_{B_1} \langle \nabla  g^i ,\nabla \varphi \rangle \mathrm{d}X = 0, $$ that is $g^i$ is harmonic in $B_1^+$. We argue similarly for $\varphi \in C_0^\infty(B_1^-).$
\end{proof}

We now prove the optimal regularity.

\begin{prop}\label{uniform.reg}
  Let $G$ be a local minimizer, then $G \in C^{0,1/2}(K;\R^m)$, for every compact set $K\subset B_1$. Moreover,
  \be\label{holder}
  |g^i(X)| \leq C_1 \mathrm{dist}\left(X, \partial \{\abs{G}>0\}\right)^{1/2} \quad\text{in $B_{1/2}$},
  \ee
  for every $i=1,\dots,m$, with $C_1>0$ a universal constant.
\end{prop}
\begin{proof}
Let $K\subset B_1$ be a compact and $i=1,\dots,m$. We claim that there exists an universal constant $C>0$ such that, for any $X_0 \in K$ and $r \in (0,1-\abs{X_0})$
\be \label{morrey}
  \frac{1}{r^{n}} \int_{B_r(X_0)}  |\nabla g^i|^2 \mathrm{d}X \leq C.
\ee
Then, by a Morrey type embedding, we deduce $g^i \in C^{0,1/2}(K)$ (see \cite{Morrey1966}). Moreover, since the constant $C>0$ is universal, the inequality \eqref{holder} is satisfied for a constant $C>0$ independent on the local minimizer.

Since by Lemma \ref{lem1}, the components of $G$ are harmonic in $B_1 \setminus \{x_{n+1}=0\}$, it is not restrictive to suppose that $X_0\in \{x_{n+1}=0\}.$ By the translation invariance of the problem, let us suppose $X_0=0$ and $r\in (0,1)$. Inspired by the proof of \cite{CRS}[Theorem 1.1.], let $\tilde{g}^i_r \colon B_r \to \R$ be the harmonic replacement of $g^i$ in $B_r$, i.e. be such that
$$
\begin{cases}
  \Delta \tilde{g}^i_r=0 & \mbox{in } B_r \\
  g^i_r=g^i & \mbox{on }\partial B_r.
\end{cases}
$$
By an integration by parts, we easily deduce
\be\label{u-uh}
\int_{B_r}{\langle \nabla \tilde{g}^i_r, \nabla (g^i- \tilde{g}^i_r)\rangle \mathrm{d}X} = 0.
\ee

Consider now the competitor $\widetilde{G_{i}} = (g^1, \dots, \tilde{g}^i_r, \dots, g^m)$. By the minimality of $G$ we get $\mathcal{J}(G,B_r) \leq \mathcal{J}(\widetilde{G_i},B_r)$, which implies
$$
\int_{B_r}{\abs{\nabla g^i}^2\mathrm{d}X} \leq
\int_{B_r}{\abs{\nabla \tilde{g}^i_r}^2\mathrm{d}X} + \omega_n r^n.
$$
Combining the ``quasi-minimality" of $g^i$ with \eqref{u-uh}, we get
$$
\int_{B_r}{\abs{\nabla (g^i-\tilde{g}^i_r)}^2\mathrm{d}X} \leq
\omega_n r^n.
$$
Thus, for $\rho \in (r,1)$ we get
\begin{align*}
\int_{B_r}{\abs{\nabla g^i}^2\mathrm{d}X} & \leq 2 \left(\int_{B_\rho}{\abs{\nabla (g^i-\tilde{g}^i_\rho)}^2\mathrm{d}X} + \int_{B_r}{\abs{\nabla \tilde{g}^i_\rho}^2\mathrm{d}X}\right)\\
&\leq C \rho^n + C \int_{B_r}{\abs{\nabla \tilde{g}^i_\rho}^2\mathrm{d}X}\\
& \leq C \rho^n + C \left(\frac{r}{\rho}\right)^{n+1} \int_{B_\rho}{\abs{\nabla \tilde{g}^i_\rho}^2\mathrm{d}X}\\
& \leq C \rho^n + C \left(\frac{r}{\rho}\right)^{n+1} \int_{B_\rho}{\abs{\nabla {g}^i}^2\mathrm{d}X},
\end{align*}
where in the third inequality we used that $|\nabla \tilde g^i_\rho|^2$ is subharmonic. Hence, fixed $\delta <1/2$ such that $q=C \delta<1$, if
$\rho=\delta^{k-1}, r=\delta^{k}$ and $\mu=\delta^n$ we get
$$
\int_{B_{\delta^{k}}} \abs{\nabla g^i}^2 \mathrm{d}X \leq C \mu^{k-1} + C \mu\delta \int_{B_{\delta^{k-1}}} \abs{\nabla g^i}^2 \mathrm{d}X
$$
and iterating the previous estimate
$$
\int_{B_{\delta^{k}}} \abs{\nabla g^i}^2 \mathrm{d}X \leq C \mu^{k-1}\sum_{i=0}^{k-1} q^i 
\leq C \mu^{k-1}\frac{1}{1-q}.
$$
Hence, there exists a universal constant $\tilde{C}>0$ such that
$$
\int_{B_{r}}{\abs{\nabla g^i}^2\mathrm{d}X} \leq \tilde{C} r^n,
$$
for every $r \in (0,1/2)$. By a covering argument we obtain the claimed inequality \eqref{morrey}.
\end{proof}
By the continuity, we immediately deduce the following corollary.
\begin{cor}\label{open}
  Let $G$ be a local minimizer, then the sets
  $$
  \{\abs{G}>0\}\quad\text{and}\quad\{ g^i >0\},\{ g^i <0\}\,\,\text{ for $i=1,\dots,m$},
  $$
  and their restrictions on $\{x_{n+1}=0\}$ are respectively open in $\R^{n+1}$ and $\R^n$.
\end{cor}

With the continuity at hands, we can easily obtain the harmonicity of the components away from $\{|G|=0\}.$
\begin{lem}\label{L-a}
  Let $G$ be a local minimizer in $B_1$. Then for every $i=1,\dots,m$ we get $$\Delta g^i=0\quad \text{ in
  $B_1^+(G)$},
  $$
  and consequently
  $$
  g^i_\pm \mbox{ is }\mbox{subharmonic in }B_1.
  $$
  Moreover $\lambda_i=\Delta g^i$ is a signed Radon measure supported on $\partial \{ \abs{G}>0\}$ with the total variation $|\Delta g^i|$ satisfying
  $$
  \langle \abs{\Delta g^i},\chi_K\rangle \leq C(n,K)\int_{B}{ |\nabla g^i|^2\mathrm{d}X},
  $$
  for every compact set $K\subset B_1$.
\end{lem}
\begin{proof}
  As in \cite{AC}, the first part of the result follows by computing the first variation of the functional $\mathcal{J}(\cdot,B_1)$ with respect to a direction $\xi e^i$, with $\xi \in C^{\infty}(\{\abs{G}>0\})$.\\ More precisely, fixed $i=1,\dots,m$, consider the competitor $G_\eps= G + \eps \xi e^i$, for some $\xi\in C^\infty_c(\{\abs{G}>0\})$ and $i=1,\dots,m$. By the previous corollary, $\{\abs{G}>0\}$ is an open set, and passing through the first variation, we get
  $$
  0=\frac12\left.\frac{d}{d\eps} \mathcal{J}(G_\eps,K)\right\lvert_{\eps=0} = \int_{\{\abs{G}>0\}\cap K}{\langle \nabla g^i,\nabla \xi\rangle \mathrm{d}X},
  $$
for every compact $K\subset B_1$.\\
Now, since $g^i_{\pm}$ are both nonnegative subharmonic in $B_1$ and $g^i$ is harmonic in $\{\abs{G}>0\}$, then $\lambda_i = \Delta g^i$ is a signed Radon measure supported in $\partial \{ \abs{G}>0\}$. Moreover, by a standard argument, let $\eta \in C^\infty_c(B_1)$
be such that $0\leq \eta \leq 1$ and $\eta \equiv 1$ on $K$. Then
$$
 \langle \abs{\Delta g^i},\chi_K\rangle \leq \langle \abs{\Delta g^i},\eta \rangle = \int_{B}{ \langle \nabla g^i, \nabla \eta \rangle \mathrm{d}X} \leq C(n,K)\int_{B}{ |\nabla g^i|^2\mathrm{d}X},
$$
as we claimed.
\end{proof}

\begin{rem}\label{Gsub}
By explicit computation we now easily deduce
$$
2 \abs{G} \Delta\left(\abs{G}\right) = \Delta(\abs{G}^2) - \frac{\abs{\nabla \abs{G}^2}^2}{2\abs{G}^2} \geq 0\quad\text{in $B^+_1(G)$},
$$
and consequently that $|G|$ is subharmonic in $B_1^+(G)$.
\end{rem}

As in the scalar case in \cite{CRS}, we can now detail the connection of global minimizer with the fractional analogue of the Bernoulli one-phase problem.
\begin{cor}\label{sharmonic}
  Let $G$ be a global minimizer in $\R^{n+1}$ and $0\in F(G)$. Then, for every $i=1,\dots,m$, the trace of $g^i$ on $\{x_{n+1}=0\}$ solves
    $$
  \begin{cases}
    (-\Delta)^{1/2} g^i(\cdot,0) = 0 & \text{in $\{\abs{G}>0\}$} \\
    g^i(\cdot,0)=0 & \text{in $\{\abs{G}=0\}$}.
  \end{cases}
  $$
\end{cor}
\subsection{Non-degeneracy}
The non-degeneracy of solutions near the free boundary points allows to obtain several results on the measure-theoretic structure of the free boundary via the blow-up analysis. We start by proving the following weak non-degeneracy condition.
\begin{prop}\label{non-deg1}
  Let $G$ be a local minimizer. Then, there exists a universal constant $c_2>0$ such that
  \be
  \abs{G}(X)\geq c_2 \mathrm{dist}(X, \partial \{\abs{G}>0\})^{1/2}\quad
  \text{in $\mathcal{B}_{1/2}^+(G)$}.
  \ee
\end{prop}
\begin{proof}
Up to translation and rescaling, it is enough to show that if $G$ is a local minimizer in a large ball and
  \be\label{assumption}
  \mathrm{dist}(0, \partial \{\abs{G}>0\})=1,
  \ee
 then $$|G|(0) \geq c_2 >0$$ for some $c_2$ small to be made precise later.
 Indeed, assume not, then $\mathcal B_1 \subset \{|G|>0\}$ and
$$g^i \quad \text{is harmonic in $B_1,$} \quad g^i(0) \leq c_2, \quad \text{for every }i=1,\ldots, m.$$  By the $C^{0,1/2}$- regularity of minimizers we deduce that the $g^i$'s are uniformly bounded say in $B_{3/4}$ and hence, since they are harmonic
$$|g^i(X) - g^i(0)| \leq K |X|, \quad \text{in $B_{1/2}$},$$ for $K>0$ universal. Thus,
$$g^i(X) \leq c_2 + K|X|, \quad \text{in $B_{1/2}$.}$$
Let
$$G_{\delta}(X) = \frac{1}{\delta^{1/2}}G(\delta X), \quad X \in B_1$$
with $\delta>0$ universal to be chosen universal later. Then, for $c_2 \leq \delta$ we get
$$g^i_\delta \leq c_2 \delta^{-1/2}+ K \delta^{1/2} \leq C\delta^{1/2} \quad\text{in $B_1$},$$
for every $i=1,\dots,m$. Moreover, since the $g^i_\delta$'s  are harmonic in $B_1$, the bound above implies
$$\|g^i_\delta\|_{L^\infty(B_{1})}, \|\nabla g^i_\delta\|_{L^\infty(B_{1/2})} \leq C \delta^{1/2}.$$
Let $\varphi \in C_0^{\infty}(B_{1/2}), 0 \leq \varphi \leq 1$ such that $\varphi \equiv 1$ in $B_{1/4}$, then
$$\int_{B_1}|\nabla g^i_\delta|^2 \mathrm{d}X \geq \int_{B_1} |\nabla (g^i_\delta(1-\varphi))|^2 \mathrm{d}X - C \delta$$
and on the other hand
$$\mathcal{L}_n(\mathcal B_1^+(|G_\delta|)) \geq \mathcal{L}_n(\mathcal B_1^+(|G_\delta|(1-\varphi))) + C_0.$$ In conclusion, by the minimality of $G_\delta$
$$0 \geq -C\delta + C_0,$$ and we reach a contradiction for $\delta$ (hence $c_2$) sufficiently small.

\end{proof}
The following result improves the non-degeneracy property of Proposition \ref{non-deg1}, and it will be fundamental in the proof of existence of non trivial blow-up limits.
\begin{prop}\label{non-deg2}
  Let $G$ be a local minimizer and $0 \in F(G)$. Then, for every $r \in (0,1/2)$
  \be\label{non-deg}
  \sup_{\mathcal B_r} \abs{G} \geq c r^{1/2},
  \ee
  for some universal constant $c >0$.
\end{prop}

In view of Proposition \ref{uniform.reg} and Remark \ref{Gsub}, Proposition \ref{non-deg2} follows immediately from the next lemma.

\begin{lem}\label{lem4} Let $v \geq 0$ be defined in $B_1$ and subharmonic in $B^+_1(v).$ Assume that there is a small constant $\eta>0$ such that \be\label{c12*}\|v\|_{C^{1/2}(B_1)}\leq \eta^{-1},\ee and $v$ satisfies the non-degeneracy condition on $\mathcal B_1$,
\be\label{c3change}
v(X) \geq \eta \; \mathrm{dist}(X,\{v=0\})^{1/2} \quad\text{for every } X \in \mathcal{B}_1.
\ee
Then if $0 \in F(v)$, we get  $$\sup_{\mathcal B_r} v \geq c(\eta)\; r^{1/2}, \quad \text{for }r \leq 1.$$\end{lem}

\begin{proof} The proof follows the lines of \cite{C3}[Lemma 7] (see also \cite{CRS}[Proposition 3.3]).\\
Given a point $X_0 \in \mathcal B_1^+(v)$ (to be chosen close to the free boundary point $0 \in F(v)$) we construct a sequence of points $(X_k)_k \subset  \mathcal B_1$ such that $$v(X_{k+1})\geq(1+\delta)v(X_k), \quad |X_{k+1}- X_k| \leq C(\eta) \mathrm{dist}(X_{k},\{v=0\}),$$ with $\delta$ small depending on $\eta$.

Then, using \eqref{c3change} 
and that $(v(X_k))_k$ grows geometrically, we find \begin{align*}|X_{k+1} - X_0| &\leq \sum_{i=0}^{k} |X_{i+1}-X_i|  \leq C(\eta) \sum_{i=0}^k \mathrm{dist}(X_{i},\{v=0\}) \\ & \leq \frac{C(\eta)}{\eta^2} \sum_{i=0}^k v^2(X_{i}) \leq c(\eta) v^2(X_{k+1}) 
.\end{align*}
Hence for a sequence of radii $r_k = \mathrm{dist}(X_k,\{v=0\}),$
we have that $$\sup_{\mathcal B_{r_k}(X_0)} v \geq c r_k^{1/2}$$ from which we obtain that
 $$\sup_{\mathcal B_{r}(X_0)} v \geq c r^{1/2}, \quad \text{for all $r \geq |X_0|.$}$$
The conclusion follows by letting $X_0$ go to $0\in F(v)$.

We now show that the sequence of $X_k$'s  exists. After scaling, assume we constructed $X_k$ such that $$v(X_k) = 1.$$ Let us call with $Y_k \in F(v)$ the point where the distance from $X_k$ to $\{v=0\}$ is achieved. By \eqref{c12*} and \eqref{c3change}, we get $$c(\eta) \leq r_k=|X_k-Y_k| \leq C(\eta).$$
Assume by contradiction that we cannot find $X_{k+1}$ in $\mathcal B_M(X_k)$ with $M$ large to be specified later, such that $$v(X_{k+1}) \geq 1+\delta.$$ Then $v \leq 1+\delta +w$ with $w$ harmonic in $B_M^+(X_k)$ and such that $$w=0 \quad \text{on $\{x_{n+1} =0\}$}, \quad w=v \quad \text{on $\p B_M(X_k) \cap \{x_{n+1}>0\}$}.$$
Thus, we have
$$w \leq C(n) \frac{x_n}{M} \sup_{B_M^+(X_k)} v \leq C \eta^{-1} x_n M^{-1/2} \leq \delta \quad \text{in $B:=B_{r_k}(X_k),$}$$ for $M$ sufficiently large depending on $\delta$. Thus, \be\label{bound1}v \leq 1+2\delta \quad \text{in $B.$}\ee
On the other hand, $v(Y_k)=0, Y_k \in \p B$. Thus from the H\"older continuity of $v$ we find \be\label{bound2}v \leq \frac 1 2, \quad \text{in $B_{c(\eta)}(Y_k)$}.\ee
If $\delta$ is sufficiently small \eqref{bound1}-\eqref{bound2} contradict that $$1=v(X_k) \leq \fint_{B} v.$$ \end{proof}
The following lemma is on the convergence of sequences
of minimizers.
\begin{lem}\label{compact}
Let $(G_k)_k$ be a sequence of local minimizer in $B_1$ uniformly bounded in $L^{2}(B_1)$. Then, up to a subsequence, there exists a limit function $G_\infty$ such that
\begin{itemize}
\item $G_\infty \in H^{1}_{\loc}(B_1)\cap C^{0,1/2}_\loc(\overline{B_1})$;
  \item $G_k \to G_\infty$ in $C^{0,\alpha}_\loc(\overline{B_1})$, for every $\alpha \in (0,1/2)$;
  \item $G_k \rightharpoonup G_\infty$ weakly in $H^{1}_\loc(B_1)$;
  \item $G_\infty$ is a local minimizer in $B_1$.
\end{itemize}
\end{lem}
\begin{proof}
By Proposition \ref{uniform.reg} we already know that $G_k \to G_\infty$ uniformly on every compact set of $B_1$ and in $C^{0,\alpha}_\loc(\overline{B})$, for every $\alpha \in (0,1/2)$. Moreover, by Ascoli-Arzel\'a theorem it follows that $G_\infty \in C^{0,1/2}(\overline{B})$. Now, let us prove that the sequence is uniformly bounded in $H^{1}_\loc(B_1)$ in order to ensure the weak convergence of sequence $G_k$. Fixed $i=1,\dots,m$ and $r \in (0,1)$, consider the competitor $G_{k,\eps}= G_k -\eps g_{k,\pm}^i \eta^2 e^i$, with $\eta \in C^\infty_c(B_{r})$ such that
$$
0\leq \eta \leq 1,\quad \eta \equiv 1\mbox{ on }B_{r/2},\quad \abs{\nabla \eta} \leq \frac{C}{r}
$$
and $\eps>0$ small enough. Note that $G_{k,\eps}= G_k$ on $\partial B_r$ and $\{\abs{G_k}>0\} =\{\abs{G_{k,\eps}}>0\}$. Therefore, from the local minimality of $G_k$ we get $\mathcal{J}(G_k,B_r) \leq \mathcal{J}(G_{k,\eps},B_r)$, which implies
$$
\int_{B_r}{\langle \nabla g^i_k, \nabla (g_{k,\pm}^i \eta^2)\rangle \mathrm{d}X} \leq
\frac{\eps}{2}\int_{B_r}{\abs{\nabla(g_{k,\pm}^i \eta^2)}^2 \mathrm{d}X}.
$$
Finally, letting $\eps \to 0$ and proceeding as in the proof of the standard Caccioppoli inequality, we deduce
\be\label{caccio}
\int_{B_{r/2}}{\abs{\nabla g^i_{k,\pm}}^2\mathrm{d}X} \leq
\frac{C}{r^2}\int_{B_r}{(g_{k,\pm}^i)^2  \mathrm{d}X},
\ee
with $C>0$ universal constant and $r\in (0,1)$. Thus, since the sequence $(G_k)_k$ is uniformly bounded in $L^2(B_1)$, by \eqref{caccio} we get that the sequence is uniformly bounded in $H^{1}_\loc(B_1)$ and it weakly converges to some $G_\infty \in H^{1}(B_1)$.

In conclusion, let us show that for every $r \in (0,1)$ we have $$\mathcal{J}(G_\infty,B_r) \leq \mathcal{J}(G_\infty+\Psi,B_r),\quad\mbox{for every }\Psi=(\psi^1,\cdots,\psi^m) \in H^{1}_0(B_r;\R^m).$$
Since we already know by Proposition \ref{uniform.reg} that there exists a local minimizer H\"{o}lder continuous of class $C^{0,1/2}$, we can assume that $\Psi$ is continuous. Therefore, for every $k>0$ let us consider the competitor
$$
G_{k,\eps} = \sum_{i=1}^m (g^i_k+\psi^i-\eps \eta)_+ e^i - (g^i_k+\psi^i+\eps \eta)_- e^i,
$$
with $\eta \in C^\infty_c(B_{(1+r)/2})$ such that $0\leq\eta\leq1$ and  $\eta \equiv 1 $ on a neighborhood of $\overline{B_r}$.\\ Hence, by the local minimality of $G_k$ in $B_{(1+r)/2}$, namely  $\mathcal{J}(G_k,B_{(1+r)/2})\leq \mathcal{J}(G_{k,\eps},B_{(1+r)/2})$, we have
\begin{align*}
\mathcal{L}_n(\mathcal{B}_{(1+r)/2}\cap \{|G_{k}|>0\} ) \leq &\, \sum_{i=1}^m\int_{B_{(1+r)/2}}{\abs{\nabla \psi^i}^2 + 2 \langle \nabla \psi^i,\nabla g^i_k\rangle\mathrm{d}X}+\\
&\,+\eps \sum_{i=1}^m\int_{\mathrm{supp}\eta \setminus B_r}{\eps\abs{\nabla \eta}^2 + 2\langle \nabla \eta,\nabla (g^i_k +\psi) \rangle\mathrm{d}X}+\\
&\, +\mathcal{L}_n(\mathcal{B}_{(1+r)/2}\cap \{|G_{k,\eps}|>0\} ).
\end{align*}
In particular, localizing the measure of the positivity set in $\mathcal{B}_r$, we get
\begin{align*}
\mathcal{L}_n(\mathcal{B}_r\cap \{|G_{k}|>0\} ) \leq &\, \sum_{i=1}^m\int_{B_{(1+r)/2}}{\abs{\nabla \psi^i}^2 + 2 \langle \nabla \psi^i,\nabla g^i_k\rangle\mathrm{d}X}+ C\eps+\\
&\,+\int_{\mathcal{B}_{(1+r)/2}}{\chi_{\{|G_{k,\eps}|>0\}}\mathrm{d}x} - \int_{\mathcal{B}_{(1+r)/2}\setminus \overline{\mathcal{B}_r}}{\chi_{\{|G_{k}|>0\}}\mathrm{d}x},
\end{align*}
where we used that $(G_k)_k$ is uniformly bounded in $H^{1}(B_{(1+r)/2})$. Since
\begin{align*}
\{g^i_k - \eps \eta>0 \}\setminus \overline{B_r} &\subseteq \{g^i_k >0\}\setminus \overline{B_r}\\
\{g^i_k + \eps \eta<0 \}\setminus \overline{B_r} &\subseteq \{g^i_k <0\}\setminus \overline{B_r}
\end{align*}
and by the uniform convergence
\begin{align*}
\{g^i_k+\psi^i -\eps >0 \}\cap\overline{B_r} &\subseteq \{g^i_{\infty} +\psi^i >0\}\cap\overline{B_r}\\
\{g^i_k+\psi^i + \eps <0 \}\cap \overline{B_r} &\subseteq \{g^i_{\infty}+\psi^i <0\}\cap\overline{B_r},
\end{align*}
we deduce\begin{align*}
\mathcal{L}_n(\mathcal{B}_r \cap \{|G_{k}|>0\}) \leq &\,\int_{B_{(1+r)/2}}{ \left(\abs{\nabla \Psi}^2 + 2 \langle \nabla \Psi,\nabla G_k\rangle\right)\mathrm{d}X}\\
&\, +\mathcal{L}_n(\mathcal{B}_r \cap \{|G_{{\infty}}+\Psi|>0\})+C\eps.
\end{align*}

Now, using that $G_k \rightharpoonup G_\infty$ weakly in $H^{1}_\loc(B_1)$ and uniformly on $\overline{B_r}$, we obtain
$$
\mathcal{J}(G_\infty,B_r) \leq \int_{B_r}{\abs{\nabla (G_\infty + \Psi)}^2\mathrm{d}X} + \mathcal{L}_n(\mathcal{B}_r \cap \{|G_{\infty}+\Psi|>0\} )+C\eps
$$
for every $\eps>0$, which implies the desired inequality.
\end{proof}

Finally, we conclude the section by proving the first corollaries of the non-degeneracy results Proposition \ref{non-deg2}. These density estimates for the positivity set of $\abs{G}$ are a obtained by a straightforward combination of the non-degeneracy condition \eqref{non-deg} and the optimal regularity of local minimizer.
\begin{cor}\label{lower.H}
Let $G$ be a local minimizer in $B_1$ and $0\in F(G)$. Then, for every $r \in (0,1/2)$ there exists $X_r \in \mathcal{B}_r$ be such that
  $$
  \mathcal{B}_{C_0r}(X_0) \subset \mathcal{B}_r^+(G),
  $$
  for some universal constant $C_0>0$. Equivalently, there exists $\eps_0>0$ such that
$$
\mathcal{L}_n(\mathcal{B}_r\cap \{|G|>0\}) \geq \eps_0 \omega_n r^n.
$$
\end{cor}
\begin{proof}
  The proof of the interior corkscrew condition is a combination of Proposition \ref{uniform.reg} and Proposition \ref{non-deg2}. More precisely, on one hand for $r$ small enough there exists $X_r\in \mathcal{B}^+_r(G)$ such that $\abs{G}(X_r)\geq C r^{1/2}$. On the other one, since $\abs{G}$ is of class $C^{0,1/2}$, by setting
  $$
  C_0 = \min\left\{1,\frac{C}{[|G|]_{C^{0,1/2}}}\right\},
  $$
  we have that $\abs{G}>0$ in $\mathcal{B}_{C_0 r}^+(|G|)$, which proves the claimed lower bound.
\end{proof}

\begin{rem}
The following estimate is a specific feature of the non-local attitude of the vectorial thin one-phase problem. Indeed, for the local case \cite[Remark 2.2]{MTV2} the authors highlight that, unlike in \cite{CSY,MTV1} where it was assumed at least one component $g^i$ to be positive, they cannot hope to have a density estimate from above on the positivity set.\\ Instead, since in our case the traces are $(-\Delta)^{1/2}$-harmonic in $\{ |G|>0\}$, the upper bound holds true thanks to the different local regularity of $(-\Delta)^{1/2}$-harmonic functions near their zero set depending on whether or not they change sign.
\end{rem}

\begin{cor}\label{upper.H}
Let $G$ be a local minimizer in $B_1$ and $0\in F(G)$. Then, for every $r \in (0,1/2)$
\be\label{above}
 \mathcal{L}_n(\mathcal{B}_r^+(G))\leq (1-\eps_0)\omega_n r^n,
 \ee
 for some universal constant $\eps_0>0$.
\end{cor}
\begin{proof}
Since $\abs{G}$ is non-negative, up to rescaling, condition \ref{above} is equivalent to
$$\mathcal{L}_n(\mathcal{B}_1\cap \{\abs{G}=0\})\geq \eps_0.$$ Thus, suppose there exists a sequence $(G_k)_k$ of local minimizers in $B_1$ such that $0 \in F(G_k)$ and
$$
\lim_{k\to \infty} \mathcal{L}_n(\mathcal{B}_1\cap\{\abs{G_k}=0\}) =0.
$$
By Proposition \ref{uniform.reg} and Lemma \ref{compact}, we already know that $G_k \to G_\infty$ weakly in $H^{1}(B_{1/2})$ and uniformly on every compact set of $B_{1/2}$. Moreover, $G_\infty\in H^{1}_{\loc}(B_{1/2}^+)\cap C^{0,1/2}_\loc(\overline{B_{1/2}})$ is a local minimizer in $B_{1/2}$. Now, let $\tilde{g}^i_k \colon B_1 \to \R$ be the  harmonic replacement of $g^i_k$ in $B_1$, i.e. be such that
$$
\begin{cases}
\Delta \tilde{g}^i_k =0 & \mbox{in } B_1 \\
  \tilde{g}^i_k=g^i_k & \mbox{on }\partial B_1.
\end{cases}
$$
By the minimality of $G_k$, given the competitor $\widetilde{G}_k = (g^1_k, \dots, \tilde{g}^i_k, \dots, g^m_k)$, from \eqref{u-uh} we deduce \be\label{convergence}
\int_{B_1}{\abs{\nabla (g^i_k-\tilde{g}^i_k)}^2\mathrm{d}X} \leq
\mathcal{L}_n(\mathcal{B}_1\cap\{\abs{G_k} =0\}) \to 0,
\ee
as $k \to \infty$. Thus, up to a subsequence, the sequence $(\widetilde{G}_k)_k$ do converge uniformly on every compact set of $B_{1/2}$ to some function $\widetilde{G}_\infty\in H^{1}_{\loc}(B_{1/2}^+)$ which is harmonic in $B_{1/2}$. Finally, by applying Fatou's Lemma to \eqref{convergence}, we get
$$
\int_{B_{1/2}}{\abs{\nabla (g^i_\infty-\tilde{g}^i_\infty)}^2\mathrm{d}X}
= 0,
$$
namely for every $i=1,\cdots,m$ we deduce that $g^i_\infty$ is harmonic in $B_{1/2}$ such that $0\in F(G_\infty)$. \\Hence, we already know that $g^i_\infty \in C^{0,\alpha}_\loc(\R^{n+1})$, for every $\alpha \in (0,1)$, in contradiction with Proposition \ref{non-deg2} for $\alpha >1/2$.
\end{proof}

\section{Weiss monotonicity formula}\label{weiss.sect}
In this section we establish a Weiss type monotonicity formula in the spirit of \cite{MTV1,MTV2}. In the case $m=1$, our result recovers the one in \cite{AP} for the scalar case. As it is well known in the literature, this result will imply the convergence of a blow-up sequence to an homogenous global minimizer.\\

For a vector-valued function $G \in H^{1}(B_1; \R^m)$, let us consider
\be
W(X_0,G,r) = \frac{1}{r^n}\mathcal{J}(G,B_r(X_0))- \frac{1}{2 r^{n+1}}\int_{\partial^+ B^+_r(X_0)}{\abs{G}^2\mathrm{d}\sigma}
\ee
The monotonicity of $r \mapsto W(X_0,g,r)$ is a fundamental tool for the classification of the blow-up limits.
\begin{thm}\label{weiss}
  Let $G$ be a local minimizer of \eqref{min.functional} and $X_0 \in F(G)$. Then, the Weiss type functional $r\mapsto W(X_0,G,r)$ is monotone non-decreasing for every $r \in (0,1-\abs{X_0})$. More precisely, we have
  \be\label{weiss.form}
  \frac{d}{dr} W(X_0,G,r) \geq \frac{1}{r^{n+2}}\sum_{i=1}^m \int_{\partial B_r(X_0)}{\left( \langle \nabla g^i, X-X_0 \rangle - \frac12 g^i \right)^2\mathrm{d}\sigma}.
  \ee
  Moreover, $W(X_0,G,\cdot)$ is constant in $(0,+\infty)$ if and only if $G$ is $s$-homogeneous with respect to $X_0$.
\end{thm}
Through the paper we will always denote with $W(X_0,G,0^+)$ the limit of the Weiss monotonicity formula as $r \to 0^+$.\\
In order to simplify the notation, since the problem is invariant under translation, in the following computations we will assume $X_0=0$ and denote $W(r)=W(0,G,r)$.
\begin{lem}\label{hom.ext}
Let $G$ be a local minimizer of \eqref{min.functional} and $0 \in F(G)$. Then, we get
\begin{align*}
\int_{B_r}{\abs{\nabla G}^2\mathrm{d}X}+ \mathcal{L}_n(\mathcal B^+_r(G)) \leq &\,\, \frac{1}{n}\int_{\partial B_r}{\left(r\abs{\nabla_{S^{n}} G}^2 + \frac14\frac{\abs{G}^2}{r} \right)\mathrm{d}\sigma}+\\
&\,+ \frac{r}{n} \mathcal{H}^{n-1}(\partial \mathcal{B}_r\cap \{|G|>0\}),
\end{align*}
for every $r\in (0,1)$.
\end{lem}
\begin{proof}
  Let us consider now the $1/2$-homogeneous extension $\widetilde{G}=(\tilde{g}_1,\cdots,\tilde{g}_m)$ of the trace of $G$ on $\partial B_r$, defined by
$$
\widetilde{G}(X)=\frac{\abs{X}^{1/2}}{r^{1/2}}G\left(X \frac{r}{\abs{X}}\right).
$$
Then, for every $i=1,\dots,m$ we get
$$
\abs{\nabla \tilde{g}_i}^2(X) = \frac14\frac{1}{r\abs{X}}g^i\left(X \frac{r}{\abs{X}}\right)^2 + \frac{r}{\abs{X}}\abs{\nabla_{S^n} g^i}^2\left(X\frac{r}{\abs{X}} \right).
$$
Integrating over $B_r^+$ and summing for $i=1,\dots,m$, we obtain
\begin{align*}
\int_{B_r}{|\nabla \widetilde{G}|^2\mathrm{d}X}
&= \int_0^r \frac{1}{\rho}\int_{\partial B_\rho}{ \left(\frac14\frac{1}{r}\abs{G}^2\left(X \frac{r}{\rho}\right) + r\abs{\nabla_{S^n}G}^2\left(X\frac{r}{\rho} \right)\right)\mathrm{d}\sigma}\mathrm{d}\rho\\
&= \frac{r}{n} \int_{\partial B_r}{\left( \frac14\frac{\abs{G}^2}{r}+r\abs{\nabla_{S^{n}} G}^2 \right)\mathrm{d}\sigma},
\end{align*}
while for the measure term we have that
$$
\mathcal{L}_n(\mathcal{B}_r\cap \{|G|>0\} ) =
\frac{r}{n} \mathcal{H}^{n-1}(\partial\mathcal{B}_r\cap \{|G|>0\} )
$$
Finally, since $\widetilde{G}= G$ on $\partial B_r$, the minimality assumption $\mathcal{J}(G,B_r) \leq \mathcal{J}(\widetilde{G},B_r)$ gives the claimed inequality.
\end{proof}
\begin{proof}[Proof of Theorem \ref{weiss}]
By the estimate of Lemma \ref{hom.ext}, we immediately get
\begin{align*}
W'(r) = &\,\frac{1}{r^n}\left( \int_{\partial B_r}{ \abs{\nabla G}^2\mathrm{d}X}+ \mathcal{H}^{n-1}(\partial \mathcal{B}_r\cap \{|G|>0\} ) \right)+\\
&\, -\frac{n}{r^{n+1}}\left(\int_{B_r}{ \abs{\nabla G}^2\mathrm{d}X}+ \mathcal{L}_n(\mathcal{B}_r\cap \{|G|>0\}  )\right) +\\
&\, - \frac{1}{r^{n+1}}\sum_{i=1}^m\int_{\partial B_r}{ g^i\partial_r g^i\mathrm{d}\sigma} +\frac{1}{2r^{n+2}}\int_{\partial B_r}{\abs{G}^2\mathrm{d}\sigma}\\
\geq &\, \frac{1}{r^n}\sum_{i=1}^m \int_{\partial B_r}{\left(\abs{\partial_r g^i}^2 - g^i\partial_r u^i + \frac{1}{4r^2}\abs{g^i}^2\right)\mathrm{d}\sigma }\\
= &\, \frac{1}{r^{n}}\sum_{i=1}^m \int_{\partial B_r}{\left( \partial_r g^i- \frac{1}{2r} g^i \right)^2\mathrm{d}\sigma}.
\end{align*}
Finally, since the right hand side of \eqref{weiss.form} is non-negative, we deduce that $W'(r)\equiv 0$ for $r \in (0,+\infty)$ if and only if
$$
\left\langle \nabla g^i(X), \frac{X}{\abs{X}}\right\rangle =\frac{1}{2\abs{X}} g^i(X)\quad \text{ in $\R^{n+1}$,}
$$
i.e. the components $g^i$ are $1/2$-homogeneous in $\R^{n+1}$.
\end{proof}
\section{Compactness and convergence of blow-up sequences}\label{blow}
This section is dedicated to the convergence of the blow-up sequences and the analysis of the
blow-up limits, both being essential for determining the local behavior of the free boundary and for the characterization of the Regular and Singular sets.\\

Let us recall the notion of blow-up sequence associated to a local minimizer $G$ in $B_1$. Given $(X_k)_k \subset F(G)$ and $r_k \searrow 0^+$ such that $B_{r_k}(X_k)\subset B_1$, we define a blow-up sequence by
\be
G_{X_k,r_k}(X) = \frac{1}{r^{1/2}_k}G(X_k+r_k X).
\ee
Then the sequence $(G_{X_k,r_k})_k$ is uniformly H\"{o}lder continuous in the class $C^{0,1/2}$ and locally uniformly bounded in $\R^{n+1}$. Thus, by Lemma \ref{compact}, up to a subsequence, $(G_{X_k,r_k})_k$ converges locally uniformly on every compact set to a function $G_0 \in H^{1}_{\loc}(B_1)\cap C^{0,1/2}_\loc(\overline{B_1})$ such that, for every $R > 0$ the following properties hold
\begin{itemize}
\item $G_{X_k,r_k} \to G_0$ in $C^{0,\alpha}_\loc(\overline{B_R})$, for every $\alpha \in (0,1/2)$;
  \item $G_{X_k,r_k} \rightharpoonup G_0$ weakly in $H^{1}(B_R)$;
  \item $G_0$ is a local minimizer in $B_R$.
\end{itemize}
Moreover, by the non-degeneracy results of the previous section, we can  guarantee the existence of a non-degenerate blow-up limit.
\begin{prop}\label{compact2}
Let $G$  be a local minimizer in $B_1$. Given $(X_k)_k \subset F(G)$ and $r_k \searrow 0^+$ such that $B_{r_k}(X_k)\subset B_1$, for every $R > 0$ the following properties hold (up to extracting a subsequence):
\begin{itemize}
  \item $G_{X_k,r_k} \to G_0$ strongly in $H^{1}(B_R;\R^m)$;
  \item the sequence of the characteristic functions $$\chi(\{|G_{X_k,r_k}|>0\})\to \chi(\{|G_{0}|>0\})$$
      strongly in
$L^1(\mathcal B_R)$;
  \item the sequence of the closed sets $\overline{\mathcal{B}^+_R(G_{X_k,r_k})}$ and its complement in $\R^n$, converge in the Hausdorff sense respectively to $\overline{\mathcal{B}^+_R(G_{0})}$ and $\R^n \setminus \overline{\mathcal{B}^+_R(G_{0})}$
  \item the blow-up limit $G_0$ is non-degenerate at zero, i.e. there exists a dimensional constant $c_0 > 0$ such that
      $$
      \sup_{\mathcal{B}_r } \abs{G_0} \geq c_0 r^{1/2}\quad\text{for every $r>0$}.
      $$
\end{itemize}
\end{prop}
\begin{proof}
For notational simplicity, we set  $G_k = G_{X_k,r_k}$.
  Since $\abs{G_{k}}$ converges locally uniformly to $\abs{G_0}$, we get
  $$
\chi(\{|G_{0}|>0\}) \leq \liminf_{k\to\infty}  \chi(\{|G_k|>0\})
  $$
 Now, let us prove that $G_k$ converges strongly in $H^{1}_\loc(\R^{n+1};\R^m)$ to $G_0$ and that the characteristic functions $\chi(\{|G_k|>0\})$ converge to $\chi(\{|G_k|>0\})$ in $L^1$. Namely, fixed a radius $R>0$, it is sufficient to prove that
  $$
  \lim_{k\to \infty}\int_{B_R}{\abs{\nabla G_k}^2\mathrm{d}X} + \mathcal{L}_n(\mathcal{B}_r^+(G_k) )=\int_{B_R}{ \abs{\nabla G_0}^2\mathrm{d}X}+\mathcal{L}_n(\mathcal{B}_r^+(G_0)).
  $$
  Consider now $\eta \in C^\infty_c(\R^{n+1}), 0\leq \eta \leq 1$ such that $\eta \equiv 1$ on $B_R$, and the competitor $\widetilde{G}_k \in H^{1}(B_1;\R^m)$ defined by
  $$
  \widetilde{G}_k = \eta G_0 + (1-\eta) G_k.
  $$
  For the sake of notational simplicity, let us set:
$$
\Omega_k = \{\abs{G_k}>0 \}\cap \R^n, \quad \widetilde{\Omega}_k=\{|\widetilde{G}_k|>0 \}\cap \R^n\quad\mbox{and}\quad \Omega_0=\{\abs{G_0}>0 \}\cap \R^n.
$$
Since $\widetilde{G}_k=G_k$ on $\{\eta =0\}$, by the optimality of $G_k$ we get
  \begin{align}\label{minim}
  \begin{aligned}
  \int_{\{\eta >0\}}{\abs{\nabla G_k}^2\mathrm{d}X} &+ \mathcal{L}_n(\Omega_k\cap \{\eta >0\})\leq
  \int_{\{\eta >0\}}{|\nabla \widetilde{G}_k|^2\mathrm{d}X} + |\widetilde{\Omega}_k\cap \{\eta >0\}|\\
  &\leq   \int_{\{\eta >0\}}{ |\nabla \widetilde{G}_k|^2\mathrm{d}X} + \mathcal{L}_n(\Omega_0\cap \{\eta =1\})+\mathcal{L}_n(\{0<\eta<1\}).
  \end{aligned}
  \end{align}
 On $\{\eta >0\}$ we calculate
 \begin{align*}
     \abs{\nabla G_k}^2 -  |\nabla \widetilde{G}_k|^2= &\,\,  \abs{\nabla G_k}^2 -  |\eta \nabla G_0 + (1-\eta)\nabla G_k +  (G_0-G_k)\nabla \eta|^2\\
=   &  \,\, (1-(1-\eta)^2)\abs{\nabla G_k}^2 - \eta^2|\nabla G_0|^2 - |G_0-G_k|^2|\nabla \eta|^2+\\
&- 2 (G_0-G_k)\langle \nabla \eta, \eta \nabla G_0 + (1-\eta)\nabla G_k \rangle - 2\eta(1-\eta)\langle \nabla G_0,\nabla G_k\rangle.
 \end{align*}
Since $G_k$ converges strongly in $L^{2}(B_R;\R^m)$ and weakly $H^{1}_\loc(\R^{n+1};\R^m)$ to $G_0$, we can estimate
\begin{align*}
\limsup_{k\to \infty}&\int_{\{\eta>0\}}{\left(  \abs{\nabla G_k}^2 -  |\nabla \widetilde{G}_k|^2\right)\mathrm{d}X}=\\
&= \limsup_{k\to \infty}\int_{ \{\eta>0\}}{ \left( (1-(1-\eta)^2) \abs{\nabla G_k}^2 -  \eta^2|\nabla G_0|^2-2\eta(1-\eta) \langle\nabla G_0,\nabla G_k\rangle\right)\mathrm{d}X}\\
&= \limsup_{k\to \infty}\int_{ \{\eta>0\}}{ (1-(1-\eta)^2) \left(\abs{\nabla G_k}^2 - |\nabla G_0|^2\right)\mathrm{d}X}\\
&\geq  \limsup_{k\to \infty}\int_{\{\eta=1\}}{ \left(\abs{\nabla G_k}^2 - |\nabla G_0|^2\right)\mathrm{d}X},
\end{align*}
where in the last inequality we used that $\abs{\nabla G_k}$ weakly converges in $L^2(\{0<\varphi<1\})$ to $\abs{\nabla G_0}$.

Combining this fact with inequality \eqref{minim}, we obtain
\begin{align*}
\limsup_{k\to \infty}&\left(\int_{\{\eta=1\}}{ \left(\abs{\nabla G_k}^2 - |\nabla G_0|^2\right)\mathrm{d}X} + \mathcal{L}_n(\Omega_k \cap \{ \eta=1\}) -\mathcal{L}_n(\Omega_0 \cap \{ \eta=1\})\right)\\
&\leq
\limsup_{k\to \infty}\left(\int_{\{\eta>0\}}{ \left(\abs{\nabla G_k}^2 - |\nabla \widetilde{G}_k|^2\right)\mathrm{d}X}+ \mathcal{L}_n(\Omega_k \cap \{ \eta=1\}) -\mathcal{L}_n(\Omega_0 \cap \{ \eta=1\})\right)\\
&\leq
\limsup_{k\to \infty}\mathcal{L}_n(\Omega_k \cap \{ \eta=1\}) -\mathcal{L}_n(\Omega_k \cap \{ \eta>0\}) + \mathcal{L}_n(\{0<\eta<1\})\\
&\leq \mathcal{L}_n(\{0<\eta<1\}).
\end{align*}
Finally, since $\eta$ is arbitrary outside $B_R$, the right hand side can be made arbitrarily small, and this implies the desired equality.\\

By Corollary \ref{lower.H} and Corollary \ref{upper.H}, we already know that
\be\label{togh}
\eps_0 \omega_n r^n \leq \mathcal{L}_n(\mathcal{B}_r\cap \{|G_k|>0\})\leq (1-\eps_0)\omega_n r^n, \quad \text{for $r<r_0/r_k$},
\ee
and for every $k>0$. Now, it is well-known that the convergence of the sequence of characteristic functions in the strong topology of $L^1$, together with \eqref{togh}, implies the Hausdorff convergence of $\overline{\Omega_k \cap B_R}$ to $\overline{\Omega_0 \cap B_R}$ locally in $\R^n$. Obviously, the same result holds for the complements  $\Omega_k^c$.\\

Finally, the non-degeneracy of the blow-up limit is a straightforward combination of the uniform convergence and the non-degeneracy condition \eqref{non-deg}. Namely, by Proposition \ref{non-deg2}, for every $k>0$ the rescaled function $G_k$ is non-degenerate in the sense
$$
\text{for every }y \in \overline{\Omega_k}, r\leq \frac{1}{2r_k} \quad \sup_{B_r(y)\cap \R^n}\abs{G_k}\geq c_0 r^{1/2}.
$$
The previous inequality is obtained by applying \eqref{non-deg} in $B_{r_kr}(y)$ for the local minimizer $G$. Finally, by the uniform convergence of $G_k$ and the Hausdorff convergence of $\Omega_k \cap B_r $ in $\R^n$, for every $y\in \overline{\Omega_0}$ we get
$$
\sup_{B_r(y)\cap \R^n}\abs{G_k}\geq c_0 r^{1/2},\,\,\text{for every }r>0.
$$
\end{proof}
The following is a straightforward application of the Weiss monotonicity formula to the blow-up limit.
\begin{cor}\label{weiss.cor}
  Let $G$ be a local minimizer and $X_0 \in F(G)$. Then every blow-up limit $G_0=(g^i_0,\cdots,g^m_0)$ of $G$ at $X_0$ is $1/2$-homogeneous in $\R^{n+1}$, i.e.
  $$
  \left\langle \nabla g^i_0(X), \frac{X}{\abs{X}}\right\rangle =\frac{1}{2\abs{X}} g^i_0(X) \,\text{ in $\R^{n+1}$},
$$
for every $i=1,\cdots,m$. Moreover, the Lebesgue density of $F(G)$ exists finite at every $X_0 \in F(G)$ and it satisfies
\begin{align}\label{density.weiss}
\begin{aligned}
\left\{\abs{G}>0\right\}^{(\gamma)} &= \left\{X_0 \in F(G) \colon \lim_{r \to 0^+} \frac{\mathcal{L}_n(\mathcal{B}_r\cap \{|G_k|>0\})}{\mathcal{L}_n(\mathcal{B}_r)}=\gamma\right\}\\ &= \left\{X_0 \in F(G) \colon W(X_0,G,0^+)=\omega_n \gamma\right\}.
\end{aligned}
\end{align}
\end{cor}
\begin{proof}
Let $X_0\in F(G)$ and $G_0$ a blow-up limit of $G$ at $X_0$ associated to a sequence $r_k \searrow 0^+$.
  By Lemma \ref{compact} and Proposition \ref{compact2}, we already know that $G_0$ is a global minimizer of the vectorial Bernoulli problem. On the other hand, by the definition of the Weiss formula, for every $\rho,r>0$ we get
  $$
  W(X_0,G,r\rho) = W(0,G_{X_0,r},\rho).
  $$
  Fixed $\rho>0$, since up to a subsequence $G_{X_0,r_k}\to G_0$ uniformly and strongly in $H^{1}(B_\rho,\R^m)$, we deduce
  $$
  W(0,G_0,\rho) = \lim_{k\to \infty}W(0,G_{X_0,r_k},\rho) =
  \lim_{k\to \infty}W(X_0,G,\rho r_k) =
  \lim_{r \to 0^+}W(X_0,G,r),
  $$
  where the last limit is unique and it does not depend on the sequence $(r_k)_k$ by the monotonicity result Theorem \ref{weiss}. Finally, since $W(0,G_0,\rho)$ is constant we get that the blow-up limit is $1/2$-homogeneous.\\
  Moreover, the homogeneity of the blow-up limits and the strong convergence of the blow-up sequences imply
  \be\label{equiva}
  \mathcal{L}_n(\mathcal{B}_1 \cap \{|G_0|>0\}) = W(0,G_0,1) = \lim_{r\to 0^+}W(X_0,G,r)=\lim_{r\to 0^+}\frac{  \mathcal{L}_n(\mathcal{B}_r(X_0) \cap \{|G|>0\})}{r^n}.
  \ee
  Hence, the density $W(X_0,G,0^+)$ coincides, up to a multiplicative constant, with the Lebesgue density of the free boundary.
\end{proof}
\begin{rem}
By \eqref{equiva}, we note that for every $X_0 \in F(G)$, the measure of the positivity set in $\mathcal B_1$ of the blow-up limit does not depend on the blow-up limit itself.
\end{rem}

\begin{rem}\label{eigen.remark}
  In the classification of the blow-up limits, we will use some results related to
eigenvalues of the Laplace-Beltrami operator
$$
\Delta_{S^n}  = \mbox{div}_{S^n}(\nabla_{S^n} ),
$$
with $\mbox{div}_{S^n}$ and $\nabla_{S^n}$ respectively the tangential divergence and gradient on $S^n$ and $x_{n+1}=r \sin(\theta_n)$. In particular, the following results hold true.

Let $\omega \subset S^{n-1}\times\{0\}$ be an open subset of the $(n-1)$-sphere and let $\Sigma_\omega = \{ r\theta \colon \theta \in \omega, r>0\}\times \{0\}$ be the cone generated by $\omega$ in $\{x_{n+1}=0\}$. Then, $g$ is a $\alpha$-homogeneous solution of
    $$
\begin{cases}
-\Delta g =0 &\emph{in } \R^{n+1}_+\\
\partial_{x_{n+1}}  g = 0& \emph{on } \Sigma_\omega \\
g = 0& \emph{on } \{x_{n+1}=0\}\setminus \Sigma_\omega,
\end{cases}
$$
    if and only if its trace $\varphi = g\lvert_{S^{n}}$ on the sphere satisfies
\be\label{lap-bel}\begin{cases}
-\Delta_{S^n} \varphi = \lambda(\alpha) \varphi  &\emph{in } S^n_+\\
\partial_{\theta_n} \varphi = 0& \emph{on } \omega \\
\varphi = 0& \emph{on } (S^{n-1}\times\{0\})\setminus \omega,
\end{cases}
\ee
with $\lambda(\alpha)=\alpha(\alpha+n-1)$ the characteristic eigenvalue associated to the section $\omega$. Moreover, both the map $\omega\mapsto \alpha(\omega)$ and  $\omega\mapsto \lambda(\alpha(\omega))$ are monotone with respect to the inclusion of spherical sets.

In particular, if  $\alpha<1$, then $\varphi$ cannot change sign and it is indeed a multiple of the principle eigenvalue. Finally, for every spherical set $\omega\subset S^{n-1}$ such that $\mathcal{H}^{n-1}(S) \leq n\omega_n/2$ we have the inequality
    $$
    \lambda_1 \geq \frac12 \left(n-\frac12 \right),
    $$
and the equality is achieved if and only if, up to a rotation, $\omega = S^{n-1} \cap \{x_n>0\}$.\\

The proof of these claims uses the monotonicity of the eigenvalue with respect to the inclusion of spherical set and the P\'{o}lya-Szeg\"{o} inequality for the Schwarz symmetrization applied to the eigenvalue problem \eqref{lap-bel} (see \cite{susste,cones} for further details).

\end{rem}

The following Lemma characterizes the structure of the blow-up limits. In particular, we can prove that the norm of every blow-up limit is a global minimizer of the scalar thin one-phase functional.
\begin{prop}\label{caract}
Let $G$ be a local minimizer and $X_0 \in F(G)$. Then, every blow-up limit $G_0$ is of the form $$G_0(X)=\xi \abs{G_0}(X)\quad\mbox{where}\quad \xi\in\R^m, \abs{\xi}=1$$ and $\abs{G_0}$ is a global minimizer of the scalar thin one-phase functional
\be\label{J}
\mathcal{J}(g,B_R) = \int_{B_R}{\abs{\nabla g}^2\mathrm{d}X} + \mathcal{L}_n(\mathcal{B}_R\cap \{g>0\} ), \quad\mbox{for }R>0.
\ee
Moreover, there exists a dimensional constant $\delta \in (0, 1/2)$ such that one of the following possibilities holds:
\begin{enumerate}
  \item[1.] The Lebesgue density of $\{|G|>0\}$ at $X_0$ is $1/2$ and every blow-up limit $G_0$ is of the form
      \be\label{blow.reg}
      G_0(X)=\xi A U(\langle x,\nu\rangle, x_{n+1})\quad\mbox{where}\quad\xi \in \R^m, \abs{\xi}=1, \nu\in S^{n-1}\times\{0\}
      \ee
      and $A>0$ a specific constant depending only on $n$.
  \item[2.] The Lebesgue density of $\{|G|>0\}$ at $X_0$ satisfies
      \be\label{dens}
      \frac12 +\delta \leq \lim_{r \to 0} \frac{|\mathcal{B}_r(X_0)\cap \{|G|>0\})|}{|\mathcal{B}_r|} \leq 1-\delta,
      \ee
      and $|G_0|$ is a nonnegative global minimizer of \eqref{J} with singularity in zero.
\end{enumerate}
\end{prop}
\begin{proof}
 Let $X_0 \in F(G)$ and $G_0$ a blow-up limit of $G$ at $X_0$. By Corollary \ref{weiss.cor} the limit $G_0$ is an $1/2$-homogeneous global minimizer such that $$|\mathcal{B}^+_1(G_0)| = \gamma|\mathcal{B}_1|,$$
for some $\gamma\in (0,1)$ (because of the density estimates).
By Lemma \ref{L-a} and Corollary \ref{sharmonic}, we have that the blow-up limit satisfies
$$
\begin{cases}
  \Delta g^i_0=0 & \mbox{in } \R^{n+1}_+ \\
  \partial_{x_{n+1}} g^i_0 =0 & \mbox{on } \{|G_0|>0\}\cap \R^n \\
  g^i_0=0 & \mbox{on }\R^n\setminus \{|G_0|>0\}.
\end{cases}
$$
Hence, in view of Remark \ref{eigen.remark} all the components are equal up to a multiplicative constant. Moreover by nondegeneracy, $G_0$ cannot be identically zero.

Thus, there exists $\xi \in \R^m$ such that $\abs{\xi}=1$ and $G_0= \xi u$, where $|G_0|=g$ and $g$ is a global minimizer of \eqref{J}. Indeed, for every $R>0$ let $\tilde{g}\in H^{1}_\loc(\R^n)$ be such that $\mbox{supp}(g-\tilde{g})\subseteq B_R$. Then, given the competitor $\widetilde{G}=\xi \tilde{g}$, we easily get that $\mathcal{J}(G_0,B_R)\leq \mathcal{J}(\widetilde{G},B_R)$ is equivalent to
$$
\int_{B_R}{\abs{\nabla g}^2\mathrm{d}X} + \mathcal{L}_n(\mathcal{B}_R\cap \{g>0\} ) \leq \int_{B_R}{\abs{\nabla \tilde{g}}^2\mathrm{d}X} + \mathcal{L}_n(\mathcal{B}_R\cap \{\tilde{g}>0\} ).
$$\\

The desired claims now follow by the known results for the scalar case, see \cite[Proposition 5.3]{DS1}.
\end{proof}

The previous analysis allows to extend the results from the scalar case to our vectorial counterpart. In particular, the problem of the existence of singular global minimizer for \eqref{min} coincides with its scalar counterpart. Indeed, by \cite{DS1} we have
$$
n^* = \inf\{k \in \N\colon \text{there exists an $1/2$-homogeneous global minimizer with singularity in zero}\}\geq 3.
$$
\begin{cor}
  Let $G_0$ be a global minimizer in $\R^{n+1}$ with $n<n^*$. Then $G_0$ is of the form \eqref{blow.reg}.
\end{cor}
Finally, we introduce the notion of regular and singular part of $F(G)$. The rest of the paper will be devoted to analyzing the smoothness of the regular part of the free boundary.
\begin{defn}\label{000}
  Let $X_0 \in F(G)$. We say that
  \begin{itemize}
    \item $X_0$ is a regular point in $\mathrm{Reg}(F(G))$, if the Lebesgue density of $\{|G|>0\}$ at $X_0$ is $1/2$;
    \item $X_0$ is a singular point in $\mathrm{Sing}(F(G))$, if $X_0 \not\in \mathrm{Reg}(F(G))$.
  \end{itemize}
\end{defn}

\section{Viscosity formulation around $\mathrm{Reg}(F(G))$}\label{visco}

In this short section we recall some basic facts about the scalar thin one-phase free boundary problem, and we state the viscosity formulation of the vector valued analogue. We show that local minimizers are indeed viscosity solutions. Hence, the analysis of the regular part of the free boundary can be performed with the viscosity methods of \cite{DR, DT}. However, as pointed out in the introduction, differently from the local case, the reduction from the vector valued  problem to the scalar one is now almost straightforward. For this reason, we start by recalling definitions and basic property for the scalar problem.

\subsection{The scalar problem} In this subsection we collect basic definitions and results for the scalar thin one-phase free boundary problem
\begin{equation}\label{FB}\begin{cases}
\Delta g = 0, \quad \textrm{in $B_1^+(g):= B_1 \setminus \{(x,0) : g(x,0)=0\} ,$}\\
\frac{\p g}{\p t^{1/2}}= 1, \quad \textrm{on $F(g):= \mathcal B_1 \cap \p_{\R^n}\{(x,0) : g(x,0)>0\}$},
\end{cases}\end{equation}
where
\begin{equation}\label{limit}\dfrac{\p g}{\p t^{1/2}}(x_0):=\di\lim_{t \rightarrow 0^+} \frac{g(x_0+t\nu(x_0),0)} {\sqrt t} , \quad \textrm{$X_0=(x_0,0) \in F(g)$},\end{equation}
with $\nu(x_0)$ the unit normal to the free boundary $F(g)$ at $x_0$ pointing toward $\mathcal{B}_1^+(g)$. For further details and proofs, we refer the reader to \cite{CRS,DR, DS1,DS2,DS3}.

First, we state the notion of viscosity solutions to \eqref{FB}, as introduced in \cite{DR}.

\begin{defn}Given $g, v$ continuous, we say that $v$
touches $g$ by below (resp. above) at $X_0 \in B_1$ if $g(X_0)=
v(X_0),$ and
$$g(X) \geq v(X) \quad (\text{resp. $g(X) \leq
v(X)$}) \quad \text{in a neighborhood $O$ of $X_0$.}$$ If
this inequality is strict in $O \setminus \{X_0\}$, we say that
$v$ touches $g$ strictly by below (resp. above).
\end{defn}

\begin{defn}\label{defsub} We say that $v \in C(B_1)$ is a (strict) comparison subsolution to \eqref{FB} if $v$ is a  non-negative function in $B_1$ which is even with respect to $x_{n+1}=0$ and it satisfies
\begin{enumerate} \item $v$ is $C^2$ and $\Delta v \geq 0$ \quad in $B_1^+(v)$;\\
\item $F(v)$ is $C^2$ and if $x_0 \in F(v)$ we have
$$v (x_0+t\nu(x_0),0) = \alpha(x_0) \sqrt t + o(\sqrt t), \quad \textrm{as $t \rightarrow 0^+,$}$$ with $$\alpha(x_0) \geq 1,$$ where $\nu(x_0)$ denotes the unit normal at $x_0$ to $F(v)$ pointing toward $\mathcal{B}_1^+(v);$\\
\item Either $v$ is not harmonic in $B_1^+(v)$ or $\alpha(x_0) >1$ at all $x_0 \in F(v).$
\end{enumerate}
\end{defn}

Similarly one can define a (strict) comparison supersolution.

\begin{defn}\label{scalar}We say that $g$ is a viscosity solution to \eqref{FB} if $g$ is a  continuous non-negative function in $B_1$ which is even with respect to $x_{n+1}=0$ and it satisfies
\begin{enumerate} \item $\Delta g = 0$ \quad in $B_1^+(g)$;\\ \item Any (strict) comparison subsolution (resp. supersolution) cannot touch $g$ by below (resp. by above) at a point $X_0 = (x_0,0)\in F(g). $\end{enumerate}\end{defn}

\subsubsection{The function $\tilde g$} In this subsection we recall the notion of $\eps$-domain variation from \cite{DR}. Via this transformation the problem \eqref{FB} can be ``linearized", as long as an appropriate Harnack type inequality is established. This is the hearth of the strategy developed in \cite{DR} and that we plan to adapt to the vectorial context.

Recall that we denote by $P$ the half-hyperplane $$P:= \{X \in \R^{n+1} : x_n \leq 0, x_{n+1}=0\}$$ and by $$L:= \{X \in \R^{n+1}: x_n=0, x_{n+1}=0\}.$$
 Also,
we call $U(X):=U(x_n,x_{n+1}),$ where $U$ is the function defined in \eqref{Unew}.

Let $g$ be a  continuous non-negative function in $\overline{B}_\rho$. We define the multivalued map $\tilde g$ which associate to each $X \in \R^{n+1} \setminus P$ the set $\tilde g(X) \subset \R$ via the formula
\begin{equation}\label{deftilde} U(X) = g(X - w e_n), \quad \forall w \in \tilde g(X).\end{equation}
We write $ \tilde g(X)$ to denote any of the values in this set.

This change of variables has the same role as the partial Hodograph transform for the standard one-phase problem. Our free boundary problem becomes a problem with fixed boundary for $\tilde g$,  and the limiting values of $\tilde g$ on $L$ give the free boundary of $g$ as a graph in the $e_n$ direction.

Recall that if g satisfies  \begin{equation}\label{flattilde}U(X - \eps e_n) \leq g(X) \leq U(X+\eps e_n) \quad \textrm{in $B_\rho,$ for $\eps>0$}\end{equation} then $\tilde g(X) \neq \emptyset$ for $X \in B_{\rho-\eps} \setminus P$  and $|\tilde g(X)| \leq \eps,$ hence  we can associate to $g$ a possibly multi-valued function $\tilde{g}$ defined at least on $B_{\rho-\eps} \setminus P$ and taking values in $[-\eps,\eps]$ which satisfies\begin{equation}\label{til} U(X) = g(X - \tilde g(X) e_n).\end{equation}
 Moreover if $g$ is strictly monotone  in the $e_n$-direction in $B^+_\rho(g)$, then $\tilde{g}$ is single-valued. See \cite[Section 3]{DR} for the basic properties of $\tilde g$.

\subsection{The Vector Valued Case.} We consider now the vector valued thin problem:
\be \begin{cases} \label{VOPnew}
\Delta G =0 & \text{in $B_1^+(|G|)$;}\\
 \frac{\partial}{\partial t^{1/2}} |G|=1 & \text{on $F(G).$}
\end{cases}\ee
Here and henceforth, for notational simplicity we use $B_1^+(G)$ in place of $B^+_1(|G|)$.

\begin{defn}\label{solution} We say that $G=(g^1,\ldots,g^m) \in C(B_1, \R^m)$ is a viscosity solution to \eqref{VOPnew} in $B_1$ if each $g^i$ is even with respect to $\{x_{n+1}=0\}$,
\be \label{visc1}\Delta g^i=0 \quad \text{in $B_1^+(G)$}, \quad \forall i=1,\ldots,m,\ee
and the free boundary condition is satisfied in the following sense. Given $X_0 \in F(G)$, and a continuous
function $\varphi$ in a neighborhood of $X_0$, then
\begin{enumerate}
\item If $\varphi$ is a strict comparison subsolution to \eqref{FB}, then for all unit directions $f$, $\langle G, f\rangle $ cannot be touched by below by $\varphi$ at $X_0.$
\item If $\varphi$ is a strict comparison supersolution to \eqref{FB}, then $\abs{G}$ cannot be touched by above by $\varphi$ at $X_0.$
\end{enumerate}
\end{defn}

\begin{rem}\label{rescale} We remark that if $G$ is a viscosity solution to \eqref{VOPnew} in $B_\lambda$, then $$G_{\lambda}(X) = \lambda^{-1/2} G(\lambda X), \quad X \in B_1$$ is a viscosity solution to \eqref{VOPnew} in $B_1.$
\end{rem}

\begin{rem}\label{sub} Notice that, if $G$ is a viscosity solution to \eqref{VOPnew}, then $|G|$ is a viscosity subsolution to the scalar thin one-phase problem \eqref{FB}. Indeed,
by the free boundary condition in Definition \ref{solution}, we easily deduce the validity of its scalar counterpart in Definition \ref{scalar} (see also Remark \ref{Gsub}).
\end{rem}

In the next proposition we prove that local minimizers are indeed viscosity solutions.

\begin{prop}\label{regolar.visc}
Let $G$ be a local minimizer in $B_1$. Then, up to a scalar multiple, $G$ is a viscosity solution of \eqref{VOPnew} in $B_1$. 
\end{prop}
\begin{proof}
Since the constant $A>0$ in \eqref{blow.reg} depends only on the dimension $n$, up to a scalar multiplication it is not restrictive to assume that $A=1$ in Theorem \ref{caract}.\\
By Lemma \ref{L-a} we already know that \eqref{visc1} is satisfied. Hence, let $\varphi$ be a strict comparison subsolution to \eqref{FB}, and suppose by contradiction that there exists a unit direction $f$ in $\R^m$ such that $\langle G,f\rangle$ is touched by below by $\varphi$ at $Y_0 \in F(G)$.\\
Consider now the blow-up sequences centered in the touching point
$$
G_{k}(X) = \frac{1}{r_k^{1/2}}G(Y_0 + r_k X) \quad\mbox{and}\quad \varphi_{k}(X) = \frac{1}{r_k^{1/2}}\varphi(Y_0 + r_k X),
$$
for some sequence of radii $r_k \to 0^+$. Up to a subsequence, they converge respectively to some $G_{0}$ and $\varphi_{0}$ uniformly on every
compact set of $\R^{n+1}$. By Definition \ref{defsub}, we get, up to rotation, that
\be \label{absurd}
\varphi_{0}(X)=\alpha U\left(x_n,x_{n+1}\right)\quad\text{with $\alpha >1$},
\ee
On the other side, by Proposition \ref{caract} the norm $|G_{0}|$ is a $1/2$-homogeneous
global minimizer of the scalar thin one-phase functional \eqref{J} such that $\{|G_0|>0\}\cap \{x_{n+1}=0\}\supset \{x_{n+1}=0,x_{n}> 0\}$. By Remark \ref{eigen.remark}, we deduce that $\{|G_0|=0\}\cap \{x_{n+1}=0\} = P$ and consequently
$$
      G_{0}(X)=\xi U\left(x_n,x_{n+1}\right)\quad\mbox{where}\quad\xi \in \R^m, \abs{\xi}=1.
$$
Hence, we immediately deduce that
$$
\varphi_0(X)= \alpha U(x_n,x_{n+1}) \leq \langle G_0, f \rangle = \langle \xi, f\rangle U(x_n,x_{n+1}),
$$
in contradiction with the hypothesis \eqref{absurd}.\\
On the other hand, let $\varphi$ be a comparison strict supersolution and let us assume that $\abs{G}$ is touched by above by $\varphi$ at some $Y_0 \in F(G)$. By the same blow-up procedure we get, up to rotation, that
$$
\varphi_{0}(X)=\alpha U\left(x_n,x_{n+1}\right)\quad\text{with $\alpha <1$},
$$
and that $|G_{0}|$ is a $1/2$-homogeneous
global minimizer of the scalar thin one-phase functional \eqref{J} such that $\{|G_0|=0\}\cap \{x_{n+1}=0\}\subset P$. As before, we get
$$
|G_{0}|(X)= U\left(x_{n},x_{n+1}\right).
$$
Since $\abs{G_0}\leq \varphi_0$, the absurd follows from the fact that $\alpha<1$.
\end{proof}

\section{Flat free boundaries: The Harnack inequality}\label{harnack.sect}

In this section we develop the basic tools for our analysis of the regular part of the free boundary. In view of Definition \ref{000}, Proposition \ref{regolar.visc}, and non-degeneracy, this boils down to understanding "flat" viscosity solutions defined below.

\begin{defn}
Let $G$ be a viscosity solution to \eqref{VOPnew} in $B_1$. We say that $G$ is $\eps$-flat in the $(f,\nu)$-directions in $B_1$, if for some unit directions $f\in \R^m, \nu \in \R^n$,
\be \label{flat} |G(X) - U(\langle x,\nu\rangle,x_{n+1}) f| \leq  \eps \quad \text{in $B_1$,}
\ee
and \be\label{nondegenra}
 |G|(x,0) \equiv 0 \quad \text{in $B_1 \cap \{\langle x, \nu \rangle  < - \eps\}$}.\ee
\end{defn}

\subsection{Key lemmas.} Below is the key proposition that allows us to reduce our analysis to the scalar case. As already remarked, this is different from the approach we followed for the local vectorial one-phase problem in \cite{DT} where we did not reduce to the scalar counterpart. In this case such reduction just requires the construction of an appropriate barrier. In the local case a "component-wise" strategy has been used by the authors in \cite{MTV2}, however it required delicate geometric measure theory tools.

\begin{prop}\label{positive1}
  Let $G$ be a viscosity solution to \eqref{VOPnew} in $B_1$. There exists $\eps_0>0$ universal such that, if G is $\eps_0$-flat in the $(f^1, e_n)$-directions in $B_1$, then
  $$
  g^1 >0 \quad\text{in $B_1^+(G)$.}
  $$
\end{prop}
\begin{proof}
Let $\eps_0>0$ to be chosen later, by the flatness assumption \eqref{flat} we deduce that
  $$
  U(X)-\eps_0 \leq g^1(X) \leq U(X) + \eps_0 \quad\text{in $B_1$}.
  $$
  For $\delta_0>0$ to be made precise later, $x_0 \in \mathcal B_1$, set $\Phi_{x_0}(X)=\delta_0(\Lambda x_{n+1}^2 - \abs{x-x_0}^2)$ with $\Lambda>0$ universal such that $\Delta \Phi >0$ in $B_1$. We aim to show that $g^1\geq \Phi_0$ on $\partial B_1^+(G)$, which implies by the comparison principle that $g^1>0$ on the $x_{n+1}$-axis minus the origin. Hence, by comparing $g^1$ with $\Phi_{x_0}$ and varying $x_0$ in $\mathcal{B}_1$ we get that $g^1>0$ in $B_1 \setminus \{x_{n+1}=0\}$. Our claim then follows by continuity.

  Clearly, on the set $\{|G|\equiv 0\}\cap \{x_{n+1}=0\}$ we have $g^1=0\geq \Phi_0$. On the other hand on $\partial B_1 \setminus (\{|G|\equiv 0\} \cap \{x_{n+1}=0\})$ we argue as follows: given $\Lambda'>\Lambda$ let
  $$
  C= \{(x,x_{n+1})\in \R^{n+1}\colon \Lambda'x_{n+1}^2- |x|^2 >0\}\supset \{\Phi>0\}
  $$
  be a slighter larger cone in $\R^{n+1}$. Since $g^1\geq U -\eps_0$, there exists an universal constant $c_0>0$ such that
  $$
  \begin{cases}
  g^1\geq c_0>0 &\text{on $\partial B_1 \cap C$}\\
    g^1\geq -\eps_0  &\text{on $\partial B_1 \setminus C$}
  \end{cases}.
  $$
  Finally, fixed
  $$
  M_0 = \max_{\partial B_1 \cap C} \Phi, \quad m_0 = \max_{\partial B_1 \setminus C} \abs{\Phi},
  $$
  let us choose $\delta_0 >0$ so that
  $$
  \delta_0 M_0 \leq c_0 \quad \text{on $\partial B_1 \cap C$}.
  $$
  Then, up to choose $\eps_0>0$ small enough, we get
  $$
  -m_0\delta_0 \leq -\eps_0 \quad\text{on $\partial B_1 \setminus C$}.
  $$ Thus, $g^1 \geq \Phi_0$ on $\p B_1^+(G)$ as desired.
\end{proof}

The following lemma allows to translate the flatness assumption
on the vector-valued function $G$ into the property that one of its components is trapped between nearby translation of a one-plane solution, while the remaining ones are small.
\begin{lem}\label{translate}
Let $G$ be a viscosity solution to \eqref{VOPnew} in $B_1$. There exists $\eps_0>0$ universal such that, if G is $\eps_0$-flat in the $(f^1, e_n)$-directions in $B_1$, then\begin{enumerate}
\item for $i=2,\ldots,m, $ \be |g^i| \leq C \eps_0 U(X+\eps_0 e_n) \quad \text{in $B_{1/2};$}\ee
\item \be U(X-C\eps_0 e_n) \leq g^1 \leq |G| \leq U(X+C\eps_0 e_n) \quad \text{in $B_{1/2}$},\ee
\end{enumerate}
with $C>0$ universal.
\end{lem}
\begin{proof}
For the bound $(i),$ let $v$ be the harmonic function in $B_1 \setminus \{X \in \mathcal{B}_1\colon x_n <-\eps_0\}$ such that
$$
v=\eps_0 \quad\text{on $\partial B_1$}, \quad v=0 \quad\text{on $\{X \in \mathcal{B}_1\colon x_n \leq -\eps_0\}$}.
$$
Since $\abs{g^i}$ is subharmonic in $B_1$ and it satisfies
$$
\abs{g^i}\leq \eps_0, \quad g^i \equiv 0 \quad\text{on $\{X \in \mathcal{B}_1 \colon x_n \leq -\eps_0\}$},
$$
by comparison principle $\abs{g^i}\leq v$ in $B_1$. Then by the boundary Harnack inequality, say for $\bar{X} = \frac12 e_n$, we deduce
$$
v(X)\leq \bar{C}\frac{v(\bar{X})}{U(\bar{X}+\eps_0 e_n)}U(X+\eps_0 e_n) \leq C \eps_0 U(X+\eps_0 e_n) \quad\text{in $B_{1/2}$},
$$
with $C>0$ universal.\\ For the bounds in $(ii),$ let $\eps_0>0$ be as in  Proposition \ref{positive1}. Then, $g^1$ is strictly positive and harmonic in $B_1^+(g_1)$ and it satisfies
$$
U(X)-\eps_0 \leq g^1(X) \leq U(X)+\eps_0 \quad\text{in $B_1$},
$$
and (for $\eps_0$) possibly smaller,
$$
\{X \in \mathcal{B}_1 \colon x_n\leq -\eps_0  \}\subset
\{X \in \mathcal{B}_1 \colon g^1=0 \}
\subset
\{X \in \mathcal{B}_1 \colon x_n\leq \eps_0  \}
$$
Thus, by \cite{DR}[Lemma 5.3.], there exists $C>0$ universal such that
$$
U(X-C\eps_0 e_n) \leq g^1(X) \leq U(X+ C\eps_0 e_n)\quad\text{in $B_{1/2}$}.
$$
According to the proof of \cite{DR}[Lemma 5.3.], since the norm $|G|$ is subharmonic in $B_1$, and $|G| =0$ on $\{x_n \leq -\eps_0\}$, the claimed bound for $|G|$ also follows. Details are omitted as they apply verbatim.\end{proof}

\subsection{Harnack Inequality}
In this subsection we state and prove a Harnack type inequality which is crucial for our method. As in the local case (see \cite[Lemma 2.4]{DT}), in the proof we use the observation that $|G|$ is a subsolution for the scalar one phase problem in $B_1$. The key difference here is that we also have that $g^1>0$, which means that the strategy of the scalar case applies straightforwardly in this context. Most details are omitted as the results of \cite{DR} can be applied directly, after observing that in their proofs it is enough for the function to be either a subsolution or a supersolution of \eqref{FB} (depending on the desired bound), or simply a positive harmonic function away from its zero set on the plate $\{x_{n+1}=0\}$.

\begin{thm}\label{harnack}
There exists a universal constant $\overline{\eps}>0$ such that, if $G$ solves \eqref{VOPnew} in $B_1$ and
\be\label{flat_2} U(X+\eps a_0 e_n) \leq g^1 \leq |G| \leq U(X+\eps b_0 e_n) \quad \text{in $B_r(X_0) \subset B_1$,}
\ee  with $$\eps(b_0 -a_0) \leq  \bar\eps r,$$
and
\be \label{smalli}|g^i|\leq r^{1/2}\left( \frac{b_0-a_0}{r}\eps \right)^{5/8}\quad \text{in $B_{1/2}(X_0)$, \quad i=2,\ldots, m,}\ee then
\be\label{flat_2harnack} U(X+\eps a_1 e_n)\leq g^1 \leq |G| \leq U(X+\eps b_1 e_n) \quad \text{ in $B_{\eta r}(X_0)$,}
\ee
with $$a_0 \leq a_1 \leq b_1 \leq b_0, \quad  b_1 - a_1 = (1-\eta)(b_0-a_0) ,$$ for a small universal constant $\eta>0$.
 \end{thm}

The following key corollary is immediately obtained. Here $\tilde{g}^1_\eps$ and $\widetilde{|G_\eps|}$ are the $\eps$-domain variations associated to $g^1$ and $|G|$ respectively and
$$
a_\eps := \left\{(X,\tilde{g}^1_\eps(X)) \colon X \in B_{1-\eps}\setminus P\right\}\quad\text{and}\quad A_\eps := \left\{(X,\widetilde{|G_\eps|}(X)) \colon X \in B_{1-\eps}\setminus P\right\}.
$$
Since domain variations may be multivalued, we mean that given $X$ all pairs $(X, \tilde g^1_\eps(X))$ belong to
$a_\eps$ for all possible values of $\tilde g^1_\eps(X)$, and similarly for $A_\eps.$

\begin{cor}\label{AA}
There exists a universal constant $\overline{\eps}>0$ such that, if $G$ solves \eqref{VOPnew} in $B_1$,
$$
U(X-\eps e_n) \leq g^1 \leq |G| \leq U(X+\eps e_n) \quad \text{in $B_1$,}
$$
and
$$|g^i|\leq \eps^{3/4}\quad \text{in $B_{1/2}$, \quad i=2,\ldots, m,}
$$
with $\eps \leq \bar{\eps}/2$ and $m_0>0$ such that ($C$ universal)
\be\label{m0}
4\eps (1-\eta)^{m_0}\eta^{-m_0} \leq \overline{\eps}, \quad \eps \leq C (1-\eta)^{5 m_0},
\ee
then the sets $a_\eps \cap (B_{1/2}\times [-1,1])$ and $A_\eps \cap (B_{1/2}\times [-1,1])$ are trapped above the graph of a function $y=a_\eps(X)$ and below the graph of a function $y=b_\eps(X)$ with
$$
b_\eps - a_\eps \leq 2(1-\eta)^{m_0-1},
$$
where $a_\eps,b_\eps$ have modulus of continuity bounded by the H\"{o}lder function $\alpha t^\beta$, with $\alpha, \beta$ depending only on $\eta$.
\end{cor}
 Indeed, we can apply repeatedly the Harnack inequality for $m=0, \ldots, m_0$ (the second inequality in \eqref{m0} guarantees that \eqref{smalli} is preserved), and obtain
 \be\label{serve}
 U(X+\eps a_m e_n) \leq g^1 \leq |G| \leq U(X+\eps b_m e_n)\quad \text{in $B_{\eta^m}$}
\ee
with $b_m-a_m = 2(1-\eta)^m$. Thus, by the properties of the $\eps$-domain variations (see \cite[Lemma 3.1]{DR}) we get
$$
a_m \leq \tilde{g}^1_\eps \leq \widetilde{|G_\eps|} \leq b_m\quad \text{in $B_{\eta^m - \eps},$}
$$
and
\begin{align*}
a_\eps \cap (B_{\eta^m - \eps} \times [-1,1]) &\subset B_{\eta^m - \eps} \times [a_m,b_m],\\
A_\eps \cap (B_{\eta^m - \eps} \times [-1,1]) &\subset B_{ \eta^m - \eps} \times [a_m,b_m],
\end{align*}
for $m=0, \ldots, m_0.$

\medskip

 We are left with the proof of the Harnack inequality, that follows easily from the next lemma.
\begin{lem}\label{fbharnack}
There exists $\eps_0>0$ universal such that  if $G$ is a solution to \eqref{VOP} in $B_1$ such that for $0<\eps\leq \eps_0$,
$$
U(X)\leq g^1(X) \leq |G|(X) \quad\text{in $B_{1/2},$}
$$
and at $\overline{X} \in B_{1/8}(\frac14 e_n)$ we have $
U(\overline{X} + \eps e_n)\leq g^1(\overline{X}),
$
then $$
U(X+\tau \eps e_n)\leq g^1(X) \leq |G|(X)\quad \text{in $B_\delta$},
$$
for universal constants $\tau,\delta >0$. Similarly, if
$$
g^1(X)\leq |G|(X) \leq U(X)\quad\text{in $B_{1/2}$}
$$
and
\be\label{vismall}|g^i|\leq \eps^{5/8} \quad \text{in $B_{1/2}$}, \quad i=2,\ldots, m,\ee
then if $g^1(\overline{X})\leq U(\overline{X} - \eps e_n),$
we get $$g^1(X)\leq |G|(X) \leq U(X-\tau \eps e_n)\quad\text{ in $B_\delta.$}$$
\end{lem}

\begin{proof}[Proof of Lemma \ref{fbharnack}] The first statement follows immediately from the fact that $g^1$ is a supersolution to \eqref{FB} hence we can apply \cite[Lemma 6.3]{DR}.

  Now, let us consider the case
  $$
g^1(X)\leq |G|(X) \leq U(X) \quad\text{in $B_{1/2}$}
$$
and
$$
|g^i|\leq \eps^{5/8} \quad \text{in $B_{1/2}$}, \quad i=2,\ldots, m.
$$ Since $|G|$ is a subsolution, in order to apply again in \cite[Lemma 6.3]{DR}, we need to check that $$|G|(\bar X) \leq U(\bar X - c\eps e_n)$$ for some $c>0$ universal.
Since $g^1(\overline{X})\leq U(\overline{X} - \eps e_n)$, we get \be\label{dr}
    g^1(\overline{X})-U(\overline{X})\leq   U(\overline{X}-\eps e_n)-U(\overline{X}) = -\partial_t U(\overline{X}-\lambda e_n)\eps \leq -c\eps,\, \lambda\in (0,\eps)
  \ee
  and
  $$
    |G|(\overline{X})-U(\overline{X})\leq g^1(\overline{X}) + C\eps^{5/4} -U(\overline{X})\leq -\frac{c}{2}\eps.
  $$ The desired bound follows arguing as in \eqref{dr}.
      \end{proof}
  We are now ready to sketch the proof of  Theorem \ref{harnack}.
  \begin{proof}[Proof of Theorem \ref{harnack}]
  Without loss of generality, let us assume $a_0=-1$ and $b_0 = 1$. Also, up to rescaling, we can take $r=1$ (hence $2\eps \leq \overline{\eps}$). Moreover, we denote with $\eps_0$ and $\delta$ the universal constants in Lemma \ref{fbharnack}, and choose $\bar \eps=\eps_0$.\\
    We distinguish two cases depending on the position of $B_r(X_0)$.\\

    {\it Case 1.} If $\mathrm{dist}(X_0,\{x_n=-\eps, x_{n+1}=0\})\leq \delta/2$ we aim to apply Lemma \ref{fbharnack}.
    Assume that for $\overline{X}=1/4 e_n$ (the other case is analogous to the scalar counterpart \cite[Theorem 6.1]{DR})
    $$
    g^1(\overline{X})\leq U(\overline{X}).
    $$
Since,

     $$
    g^1 \leq |G| \leq U(X+\eps e_n) \quad \text{in $B_{1/2}(-\eps e_n) \subset B_1(X_0)$,}
    $$
 and  for $\eps$ small enough, it holds $\overline{X} \in B_{1/8}((-\eps + 1/4)e_n)$, by \eqref{smalli} we can apply Lemma \ref{fbharnack} and conclude that
    $$
    g^1 \leq |G| \leq U(X+(1-\eta)\eps e_n) \quad \text{in $B_{\delta}(-\eps e_n)$.}
    $$
    Finally, the improvement follows by choosing $\eta < \delta/2$, which implies that $B_\eta(X_0) \subset B_\delta(-\eps e_n)$.\\

    {\it Case 2.}   If $\mathrm{dist}(X_0,\{x_n=-\eps, x_{n+1}=0\}) > \delta/2$, then we can apply directly \cite[Theorem 6.1]{DR}, as in this case we only use that $g^1$ is a positive harmonic function in $B_1^+(g^1)$, thus the conclusion
    \be\label{flat_2harnack2} U(X+\eps a_1 e_n)\leq g^1 \leq U(X+\eps b_1 e_n) \quad \text{ in $B_{\eta}(X_0)$,}\ee does hold for $\eta$ small.
    On the other hand, reasoning as in Lemma \ref{translate}-(i) we have in the same ball,
    $$|G| \leq U(X+\eps b_1 e_n) + C \eps^{5/8}U(X+\eps e_n) \leq U(X+\bar b_1\eps e_n),$$
    and our claim is proved.
      \end{proof}

   \section{The improvement of flatness lemma}\label{final}

In this section we prove our main lemma, from which the $C^{1,\alpha}$ regularity of a flat free boundary follows by standard arguments (see for example \cite{DT}). In view of Lemma \ref{translate} the flatness can be expressed as in \eqref{flat1}-\eqref{non_d1}.

\begin{lem}\label{IMPF}[Improvement of flatness] Let $G$ be a viscosity solution to \eqref{VOPnew} in $B_1$ with $0 \in F(G)$, satisfying
\be\label{flat1}  U(X-\eps e_n) \leq g^1 \leq |G| \leq U(X+\eps e_n)  \quad\text{in $B_1$},
\ee
and
\be\label{non_d1} |G-g^1f^1| \leq \eps^{3/4} \quad \text{in $B_1$}.\ee If  $0<\rho \leq \rho_0$ for a universal $\rho_0>0$ and $0<\eps \leq \eps_0$ for some $\eps_0$ depending on $\rho$, then for unit vectors $\nu \in \R^n$ and $\bar f^1 \in\R^m$,
\be\label{flat_imp}
 U\left(\langle x, \nu\rangle -\frac{\eps}{2}\rho, x_{n+1}\right) \leq G \cdot \bar f^1 \leq |G| \leq U\left(\langle x, \nu\rangle +\frac{\eps}{2}\rho, x_{n+1}\right)\quad\text{in $B_\rho$},
\ee
and
\be\label{non_dimpr} |G - (G \cdot \bar f^1)\bar f^1| \leq \left(\frac \eps 2\right)^{3/4} \rho^{1/2} \quad \text{in $B_\rho$},\ee
with $|\nu -e_n|, |f^1-\bar f^1 |\leq C\eps, i=1,\ldots,m$, for a universal constant $C>0.$ 
\end{lem}
\begin{proof}
Following the strategies of \cite{DT, DR}, we proceed by contradiction. Once the argument reduces to the scalar case, we omit the details and refer the reader to the corresponding steps in the proof of \cite[Theorem 7.2]{DR}.\\

{\it Step 1 - Compactness and linearization.} Fix $\rho \leq \rho_0$ to be made precise later. Let us suppose there exist $\eps_k \to 0$ and a sequence of solutions $(G_k)_k$ of \eqref{VOPnew} such that $0 \in F(G_k)$ and \eqref{flat1} and \eqref{non_d1} are satisfied for every $k$,
\be\label{contrad1}
U(X-\eps_k e_n) \leq g^1_k \leq |G_k| \leq U(X+\eps_k e_n) \quad \text{in $B_{1}$,}
\ee
and
\be\label{contrad2}
|G_k - g^1_kf^1| \leq  \eps_k^{3/4} \quad \text{in $B_{1},$}\ee but  either of the conclusions \eqref{flat_imp} or \eqref{non_dimpr} does not hold.
 Let $\tilde{g}^1_k$ and $\widetilde{|G|_k}$ be the $\eps_k$-domain variations of $g^1_k$ and $|G_k|$ respectively.
In view of \eqref{contrad1}-\eqref{contrad2}, we can apply Corollary \ref{AA} and Ascoli-Arzel\`{a} to conclude that, up to a subsequence, the sets
$$
a_k := \left\{(X,\tilde{g}^1_k(X)) \colon X \in B_{1-\eps_k}\setminus P\right\}\quad\text{and}\quad A_k := \left\{(X,\widetilde{|G_k|}(X)) \colon X \in B_{1-\eps_k}\setminus P\right\},
$$
converge uniformly, with respect to the Hausdorff distance, in $B_{1/2}\setminus P$ to the graphs
$$
a_{\infty} := \left\{(X,\tilde{g}^1_{\infty}(X)) \colon X \in B_{1/2}\setminus P\right\}\quad\text{and}\quad A_{\infty} := \left\{(X,\widetilde{|G_{\infty}|}(X)) \colon X \in B_{1/2}\setminus P\right\},
$$
with $\tilde{g}^1_{\infty}$ and $\widetilde{|G_{\infty}|}$ H\"{o}lder continuous functions in $B_{1/2}$. Moreover,
\be\label{equal}\widetilde{|G_\infty|} \equiv \tilde{g}_\infty^1 \quad \text{in $B_{1/2}.$}\ee

Since $g^1_k$ is a sequence of supersolutions to the scalar thin one-phase problem \eqref{FB}, while $|G_k|$ is a sequence of subsolutions to the same problem, we conclude by the arguments in Step 2 of \cite[Theorem 7.1]{DR} that $\widetilde{|G_\infty|} \equiv \tilde{g}_\infty^1$ satisfies (in the viscosity sense) the linearized problem
\be\begin{cases}\label{linearized}
 \Delta(U_n w)=0 & \text{in $B_1\setminus P,$}\\
|\nabla_r w|=0 &\text{on $B_1\cap L$},
\end{cases}\ee
where we recall that $$
|\nabla_r w|(X_0)= \lim_{(x_n,x_{n+1})\to (0,0)} \frac{w(x_0',x_n,z)-w(x_0',0,0)}{r},\quad r=|(x_n,x_{n+1})|.
$$  In particular, since $(\tilde{g}^1_k)_k$ and $(\widetilde{|{G}_k|})_k$ are uniformly bounded in $B_1$, we get a uniform bound on $\tilde g^1_\infty \equiv \widetilde{|{G}_\infty|}$,
hence by \cite[Lemma 4.2]{DR}, since $\tilde{g}^1_\infty(0)=0,$ we deduce that for $C_0$ universal,
\be\label{old}
     \abs{\tilde{g}^1_\infty(X) -\langle \xi', x'\rangle}\leq C_0\rho^{3/2} \quad\text{in $B_{2\rho}$},
\ee
  for some vector $\xi' \in \R^{n-1}$. Details are omitted as we reduced to the scalar case, hence the arguments of \cite[Theorem 7]{DR} apply verbatim.

\smallskip

{\it Step 2 - Improvement of flatness.} In view of \eqref{old},
for $\rho< 1/(8C_0)$ small enough, we get
  $$
   \langle \xi', x'\rangle - \frac18\rho\leq \tilde{g}^1_\infty(X) \leq \langle \xi', x'\rangle +\frac18\rho \quad\text{in $B_{2\rho}$}
  $$
  and, for $k$ sufficiently large, we deduce from the uniform convergence of $a_k$ to $a_\infty$ and of $A_k$ to $A_\infty$ that
\be \label{k}
   \langle \xi', x'\rangle - \frac14\rho\leq \tilde{g}^1_k(X) \leq \widetilde{|G_k|}(X)\leq \langle \xi', x'\rangle +\frac14\rho \quad\text{in $B_{2\rho} \setminus P$}.
   \ee
The argument of Step 2 in \cite[Theorem 7.1]{DR} then gives (again details are omitted):
\be\label{flat_imp2}
 U\left(\langle x, \nu\rangle -\frac{\eps_k}{4}\rho, x_{n+1}\right) \leq g^1_k \leq |G_k| \leq U\left(\langle x, \nu\rangle +\frac{\eps_k}{4}\rho, x_{n+1}\right)\quad\text{in $B_{\frac 3 2 \rho}$},
\ee
for a unit vector $\nu$ with $|\nu-e_n| \leq C\eps_k.$

On the other hand, by \eqref{contrad1}-\eqref{contrad2} we conclude that, up to a subsequence, $g^i_k/\eps_k^{3/4} \to g^i_*$ uniformly, with $g^i_*$ harmonic in $B_{1/2} \setminus P$ and $g^i_*=0$ on $L\cap B_{1/2}$, $i=2, \ldots, m$. Thus, for $k$ large, $|M_i| \leq M$ universal, 
$$|g^i_k - M_i U \eps_k^{3/4}| \leq C \eps_k^{3/4} \rho U \quad \text{in $B_{\frac 3 2 \rho}$}.$$
From the properties of the function $U$ and \eqref{contrad1}, we conclude that ($C$ universal)
\be\label{small}|g^i_k - M_i g^1_k \eps_k^{3/4}| \leq C\eps_k^{3/4}(\rho^{3/2}+ \eps_k^{1/2})\leq (\frac{\eps_k}{8})^{3/4} \rho^{1/2} \quad \text{in $B_{\frac 3 2\rho}$},\ee by choosing $\rho\leq \rho_0$ small enough universal and then $k$ large.

Now, set
$$\xi_k^1:= f^1 + \eps_k^{3/4} \sum_{i \neq 1} M_i f^i, \quad \bar f^1_k: = \frac{\xi^1_k}{|\xi_k^1|}.$$ Notice that,
\be\label{den}\bar f^1_k = \xi_k^1 + O(\eps_k^{3/2}).\ee
We claim that
\be\label{xi1} U\left(\langle x, \nu\rangle -\frac{\eps_k}{2}\rho, x_{n+1}\right) \leq G_k \cdot \bar f^1_k \leq |G_k| \leq U\left(\langle x, \nu\rangle +\frac{\eps_k}{2}\rho, x_{n+1}\right)\quad\text{in $B_\rho$},\ee
and
\be\label{xi2} |G_k - (G_k \cdot \bar f^1_k)\bar f^1_k| \leq  (\frac{\eps_k}{2})^{3/4} \rho^{1/2} \quad \text{in $B_\rho$}, \ee
thus reaching a contradiction.
Indeed, the upper bound in \eqref{xi1} is also a straightforward consequence of \eqref{flat_imp2}. For the lower bound, we observe that by \eqref{contrad1}, \eqref{contrad2} and \eqref{den},
$$|G_k -  U \bar f^1_k| \to 0, \quad \text{as $k\to \infty$},$$
while
\be\label{fbi}|G_k| \equiv 0 \quad \text{in $\{x_n \leq -\eps_k\}$}.\ee
Proposition \ref{positive1} then gives
\be\label{pos}G_k \cdot \bar f^1_k >0 \quad \text{in $B^+_{\frac 1 2}(G_k)$}.\ee
Moreover, by the definition of $\bar f^1_k$, \eqref{den} and \eqref{pos},
\be\label{bou}G_k \cdot \bar f^1_k \geq  (U\left(\langle x, \nu\rangle -\frac{\eps_k}{4}\rho, x_{n+1}\right) -C \eps_k^{3/2})^+.\ee Call
$$h(X):=(U\left(\langle x, \nu\rangle -\frac{\eps_k}{4}\rho, x_{n+1}\right) -C \eps_k^{3/2})^+.$$

Let $V$ be the harmonic function in $B_{\frac 3 2 \rho} \setminus \{\langle x, \nu \rangle \leq \frac{\eps_k}{4}\rho\}$ with
$$V=h \quad \text{on $\p B_{\frac 3 2 \rho}$}, \quad V=0 \quad \text{on $\{\langle x, \nu \rangle = \frac{\eps_k}{4}\rho\}$.}$$ Then, by \eqref{flat_imp2}-\eqref{pos}-\eqref{bou} and the comparison principle, we conclude that $$G_k \cdot \bar f^1_k \geq V \quad \text{in $B_{\frac 3 2 \rho}$}.$$
On the other hand, by Boundary Harnack,
$$V \geq (1-C\eps_k^{3/2})U\left(\langle x, \nu\rangle -\frac{\eps_k}{4}\rho, x_{n+1}\right) \quad \text{on $B_{\rho},$}$$
for $C>0$ universal, from which the required lower bound follows for $k$ large.

We are left with the proof of \eqref{xi2}. In view of \eqref{den}, we need to show that
$$ |G_k - (G_k \cdot \xi^1_k)\xi^1_k| \leq  \left(\frac{\eps_k}{4}\right)^{3/4} \rho^{1/2} \quad \text{in $B_\rho$}.$$ Call $$\bar G_k:=G_k - (G_k \cdot \xi^1_k)\xi^1_k. $$ Then,
$$|\bar g_k^1|= \eps^{3/4} |\sum_{i \neq 1} M_i g^i| \leq C \eps_k^{3/2},$$
in view of assumption \eqref{contrad2}. For the remaining components we use \eqref{small}, hence
$$|\bar g^i_k| = |g_k^i - \eps_k^{3/4}M_i g^1_k- \eps_k^{3/2} M_i\sum_{j\neq 1}M_j g^j_k| \leq \left(\frac{\eps_k}{8}\right)^{3/4} \rho^{1/2} + C \eps_k^{9/4},$$
and the desired bound follows for $k$ large.
\end{proof}

The proof of our main result Theorem \ref{mmm} now follows combining Proposition \ref{000} (and its corollary), Definition \ref{000}, Proposition \ref{regolar.visc}, and Theorem \ref{main.visc}. The statements about $n^*$ and the fact that  $\{|G|>0\} \cap \{x_{n+1}=0\}$ has locally finite perimeter follow exactly as in the scalar case (see \cite[Section 5]{DS1} and \cite[Theorem 1.2.]{EKPSS}).
\bibliography{bibliography}
\bibliographystyle{abbrv}
\end{document}